\documentclass[11pt]{amsart}
\usepackage{a4wide,enumerate,enumitem,amsmath,color}
\usepackage[utf8]{inputenc}
\usepackage{amssymb}
\usepackage{amsthm}
\usepackage{graphicx}
\usepackage{subcaption}
\usepackage[noadjust]{cite}
\usepackage{bigints}
\usepackage{breqn}
\usepackage{tikz}
\usepackage{mathtools}
\usepackage{hyperref}
\usepackage{cite}
\usepackage{dsfont}
\usepackage {float}
\usepackage{caption}
\newtheorem{theorem}{Theorem}[section]

\newtheorem{lemma}{Lemma}
\newtheorem{proposition}{Proposition}
\newtheorem{remark}{Remark}

\newcommand{\dd}{\mathrm{d}}
\newcommand{\pa}{\partial}

\newcommand{\R}{\mathbb{R}}

\newcommand{\N}{\mathbb{N}}
\newcommand{\eps}{\varepsilon}

\newcommand{\rhosc}{\rho_{{\rm s}}}
\newcommand{\rhopc}{\rho_{{\rm p}}}
\newcommand{\rhoent}{\rho_{{\rm e}}}
\newcommand{\rhogc}{\rho_{{\rm g}}}

\newcommand{\rhose}{\rho^\eps_{{\rm s}}}
\newcommand{\rhope}{\rho^\eps_{{\rm p}}}
\newcommand{\rhoee}{\rho^\eps_{{\rm e}}}
\newcommand{\rhoge}{\rho^\eps_{{\rm g}}}

\newcommand{\fsc}{f_{{\rm s}}}
\newcommand{\fpc}{f_{{\rm p}}}
\newcommand{\fent}{f_{{\rm e}}}
\newcommand{\fgc}{f_{{\rm g}}}

\newcommand{\T}{\mathcal{T}}

\newcommand{\Esc}{\mathcal{E}_{\rm s}}
\newcommand{\Epc}{\mathcal{E}_{\rm p}}
\newcommand{\Eent}{\mathcal{E}_{\rm e}}
\newcommand{\Egc}{\mathcal{E}_{\rm g}}

\newcommand{\vsc}{v_{\rm s}}

\newcommand{\wse}{w^\eps_{{\rm s}}}
\newcommand{\wpe}{w^\eps_{{\rm p}}}
\newcommand{\wee}{w^\eps_{{\rm e}}}
\newcommand{\wge}{w^\eps_{{\rm g}}}

\begin{document}

\title[Analysis of a degenerate parabolic system]{Analysis of a degenerate parabolic system for cell dynamics in intestinal crypts}

\author[A. El Hajj]{Ahmad El Hajj}
\author[M. EL HAJJ CHEHADE]{Mohamad El Hajj Chehade}
\author[A. Zurek]{Antoine Zurek}
\address{Universit\'e de Technologie de Compi\`egne, LMAC, 60200 Compi\`egne, France}
\email{elhajjah@utc.fr,mohamad.el-hajj-chehade@utc.fr,antoine.zurek@utc.fr}

\date{\today}

\thanks{The authors thank B. Laroche for fruitful discussions on the biological meaning of the model introduced in~\cite{DHCLL22,darri2020} and the derivation of a dimensionless version of this system.}

\begin{abstract}
In this work, we study a system of degenerate parabolic equations modeling the dynamics of multiple cell populations in intestinal crypts. The model describes cell division, differentiation, and migration through a strongly coupled system of reaction-cross-diffusion equations with degenerate diffusion. By working with initial data in $BV$, we first consider a regularized form of the system and establish uniform $BV$ estimates. Using these bounds, we then pass to the limit to obtain the existence of weak solutions.

\bigskip
		

\medskip
		

\end{abstract}

\maketitle
\tableofcontents

\section{Introduction}

In this paper, we analyze a system of PDEs inspired by the one introduced in~\cite{DHCLL22,darri2020} to model the dynamics of cells in colonic crypts. Crypts are microscopic invaginations of the intestinal epithelium, looking like tiny vertical pits in the intestinal mucosa.  The shape of a crypt can be broadly approximated by that of a test-tube. The crypt is lined by several cell types, restricted here to five main categories, namely stem cells (s), deep crypt secretory cells (dcs), progenitor cells (p), enterocytes (e), and goblet cells (g). In particular, as represented in Figure~\ref{Fig1}, the crypts are organized in three main areas: stem and dcs cells at the bottom, progenitor cells in the middle, and enterocyte and goblet cells at the top of the crypts. It should be understood that Figure~\ref{Fig1} is only a schematic representation of these zones and in reality the zones are not strongly segregated, i.e., at each interface we can find a mixture of different types of cells.

\begin{figure}[!ht]
\includegraphics[scale=0.7]{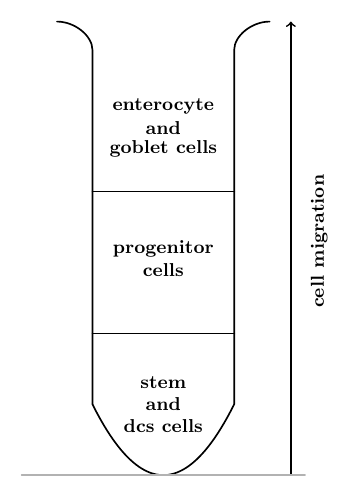}
\caption{Schematic representation of the repartition of the cell colonies inside the colonic crypt.}\label{Fig1}
\end{figure}

These crypts play a key role in the production of epithelial cells. Indeed, these epithelial cells are produced at the bottom of the crypt by stem cells and are removed at the top of the crypt by extrusion. More precisely, the stem cells, located at the bottom of the crypt, can create other stem cells by division and can also produce progenitor cells by differentiation. These progenitor cells can also produce, by division or differentiation, progenitors, enterocyte and goblet cells. Finally, the enterocyte and goblet cells are extruded when they reach the top of the crypt. Therefore, along the way, the epithelial cells move from the bottom to the top of the crypt by migration. This migration is due to the ``pressure'' exerted by all the cells present in the crypt. Besides, in~\cite{DHCLL22,darri2020}, the dcs cells are seen as a sedentary colony of cells which ensures a favorable environment for the stem cells. Eventually, as described in~\cite{DHCLL22,darri2020}, the concentrations of butyrate and oxygen have an impact on the evolution of epithelial cells inside the crypts.
 
From a mathematical point of view, the original model introduced in~\cite{DHCLL22, darri2020} is a cross-diffusion system coupled with a reaction-diffusion system. The cross-diffusion system models the evolution of the densities of cells present in the crypts, while the reaction-diffusion system models the evolution of the concentrations of oxygen and butyrate. The coupling of the systems is made through the reaction terms. The derivation of this PDE system is obtained as follows. First, a cell centered Piecewise Deterministic Markov Process (PDMP)~\cite{Azais_reviewPDMP14,Davis84} models the evolution of cells at the microscopic level, while the evolution of the concentrations of oxygen and butyrate is modelled at the macroscopic level thanks to a reaction-diffusion system. Then, in a second step, it is shown that in the mean field limit, this microscopic model tends to a deterministic (macroscopic) nonlocal system of PDEs on the densities of cells in the crypts, still coupled with the reaction-diffusion system on the chemical concentrations. Finally, a formal localization limit is performed to derive the system of local PDEs.

On the one hand, the PDMP model allows a fine description at the microscopic level of each cell in the crypts. However, its simulation cost is prohibitive and such a model is not suitable for intensive simulation. This restricts its use for applications. On the other hand, even if the macroscopic PDEs system derived from the PDMP model is less accurate, this system is an appropriate approximation of the microscopic one as shown numerically in~\cite{DHCLL22,darri2020}, and, of course, its computational cost is more affordable. Therefore, our main objective in this first work is to study at the theoretical level a (simplified) version of this PDE system. This will allow us in future works to develop and analyse an efficient and well-adapted numerical scheme.


\subsection{Outline of the paper}

The paper is organized as follows. Section~\ref{sec2} is dedicated to the presentation of the mathematical model and our existence result. Then, we show this result in Section~\ref{subsec : regul} and Section~\ref{sec : exi proof}. We first introduce and study, in Section~\ref{subsec : regul}, a convenient regularized system. Next, in Section~\ref{sec : exi proof}, we establish some uniform estimates and we pass to the limit on this regularized system. This will allow us to conclude the proof of our main result.


\section{Mathematical model and main result}\label{sec2}

\subsection{Introduction of the mathematical model}

In this first study, we will focus on the evolution, on $\Omega := (0,1)$, of the densities of $\rhosc$, $\rhopc$, $\rhoent$ and $\rhogc$ and the concentration of butyrate $c_{\rm b}$. In particular, we will neglect the densities of dcs cells as well as the concentration of oxygen. We introduce the set $\T := \{{\rm s},{\rm p}, {\rm e}, {\rm g}\}$. Then, the system writes as follows
\begin{align}\label{eq : rhosc}
\pa_t \rhosc - \pa_x(\rhosc \pa_x \rho) &= \fsc\left(x,\rho,\rhosc, c_{\rm b}\right), \quad \,\,\,\quad\mbox{in }\Omega \times (0,T),\\
\pa_t \rhopc - \pa_x(\rhopc \pa_x \rho) &= \fpc\left(x,\rho,\rhopc,\rhosc, c_{\rm b}\right), \quad \mbox{in }\Omega \times (0,T),\label{eq : rhopc}\\
\pa_t \rhoent - \pa_x(\rhoent \pa_x \rho) &= \fent\left(x,\rho,\rhoent,\rhopc, c_{\rm b}\right), \quad \mbox{in }\Omega \times (0,T),\label{eq : rhoent}\\
\pa_t \rhogc - \pa_x(\rhogc \pa_x \rho) &= \fgc\left(x,\rho,\rhogc,\rhopc, c_{\rm b}\right), \quad \mbox{in }\Omega \times (0,T),\label{eq : rhogc}\\
\pa_t c_{\rm b} - \sigma_{\rm b} \pa_x^2 c_{\rm b} &= \gamma \dfrac{c_{\rm b} + c_{\rm b}^d}{1+c_{\rm b} + c_{\rm b}^d} \, (\rhoent + \rhogc), \,\quad \, \mbox{in } \Omega \times (0,T)\label{eq : cb},
\end{align}
where $\rho = \sum_{i\in\T} \rho_i$ denotes the total density of cells and $\sigma_{\rm b}$, $\gamma$, and $c_{\rm b}^d$ are positive constants. 

Moreover, in the initial model introduced in~\cite{DHCLL22,darri2020}, a correction function $\phi$ only depending in space is included in the definition of the fluxes of each cell densities, i.e., the fluxes are of the form $-\phi \rho_i \pa_x \rho$ for all $i \in \T$. This function is nonnegative and vanishes at $x=0$ and $x=1$. In fact, in the microscopic model, the authors of~\cite{DHCLL22,darri2020} assume that the crypt admits a cylinder symmetry around its vertical axis so that the crypt can be represented as a two dimensional symmetric domain (as in Figure~\ref{Fig1}). Then, the center of each cell in the crypt is projected on a one dimensional axis, and the dynamics of these projected centers is modeled. The function $\phi$ is introduced to take into account the impact of the curvature at the bottom and top of the crypt on the mechanical interaction between cells. In this paper, we neglect this correction function. This system is complemented by the following boundary and initial conditions:
\begin{align}
\pa_x \rho_i(0,t) &= \pa_x \rho_i(1,t) = 0 \quad \mbox{for } t\in (0,T) \mbox{ and } i \in \T,\label{boundary rhoi}\\ 
\pa_x c_{\rm b}(0,t) &= 0, \quad c_{\rm b}(1,t) = 0 \quad \mbox{for } t \in (0,T),\label{boundary cb}\\
\rho_i(x,0) &= \rho_i^0(x) \quad \mbox{for } x \in \Omega \mbox{ and }  i \in \T,\\
c_{\rm b}(x,0) &= c_{\rm b}^0(x) \quad \mbox{for } x \in \Omega.\label{initial cb}
\end{align}

\begin{remark}
    In~\cite{DHCLL22,darri2020}, the function $c_b$ satisfies at $x=1$ the non-homogeneous Dirichlet boundary condition $c_{\rm b}(1,t) = c_{\rm b}^d$ and the source term in~\eqref{eq : cb} is replaced by $\gamma \,c_{\rm b}\,(\rhoent+\rhogc)/(1+c_{\rm b})$. Of course, introducing the function $\widetilde{c}_{\rm b} = c_{\rm b} + c^d_{\rm b}$, both formulation are equivalent.
\end{remark}

Before giving the precise definitions of the source terms of the above system, let us notice that cross-diffusion systems of the type~\eqref{eq : rhosc}--\eqref{eq : rhogc} intervene, for instance, in dynamic cell populations~\cite{MurTog15,CHS18} or cancer invasion models~\cite{ChaLol05,Preziosi03,APS06}. Indeed, this type of system is well known to describe interacting biological species that disperse in response to population pressure.  Here, as for instance described in~\cite{GuPi84}, it is assumed that the population pressure function $p$ is given by Darcy's law, i.e., $p(\rho)=-\pa_x\rho$. However, more general definition of $p$ can be used to model different phenomena, see Section~\ref{sec : difficulties}.

Now, following~\cite{DHCLL22}, the source terms are defined as
\begin{align}\label{def : fsc}
\fsc(x,\rho,\rhosc, \, c_{\rm b} ) &= \rhosc \, q_{{\rm div,s}} \, \big(1-R_{{\rm div,s}}(x)\big)\,\big(1-\overline{R}_{{\rm div,s}}(\rho)\big)\big(1-\underline{R}_{{\rm div,s}}(c_{\rm b})\big)\\ 
&- \rhosc \, q_{{\rm s,p}} \, R_{\rm s, p}(x),\nonumber\\
\fpc(x,\rho,\rhopc,\rhosc, c_{\rm b} ) &= \rhopc\, q_{{\rm div,p}}\, \big(1-R_{{\rm div,p}}(x)\big) \, \big(1-\overline{R}_{{\rm div,p}}(\rho)\big)\label{def : fpc}\\
&- \rhopc \, q_{{\rm p,e}} \, R_{{\rm p,e}}(x) \underline{R}_{{\rm p,e}}(c_{\rm b})\nonumber\\
&- \rhopc \, q_{{\rm p,g}} R_{{\rm p,g}}(x) \underline{R}_{{\rm p,g}}(c_{\rm b}) + \rhosc \, q_{{\rm s,p}} R_{{\rm s,p}}(x),\nonumber\\
\fent(x,\rho,\rhoent, \rhopc,c_{\rm b}) &= \rhopc \, q_{{\rm p,e}} R_{{\rm p,e}}(x) \underline{R}_{{\rm p,e}}(c_{\rm b}) - \rhoent \, q_{{\rm ex,e}} \, R_{{\rm ex,e}}(x) \, \overline{R}_{{\rm ex,e}}(\rho),\label{def : fent}\\
\fgc(x,\rho,\rhogc,\rhopc, c_{\rm b} ) &= \rhopc \, q_{{\rm p,g}} R_{{\rm p,g}}(x) \underline{R}_{{\rm p,g}}(c_{\rm b}) - \rhogc \, q_{\rm ex, g} \, R_{\rm ex, g}(x) \, \overline{R}_{\rm ex, g}(\rho),\label{def : fgc}
\end{align}
where $q_{.,.}$ are given positive constants and the functions $R_{\cdot,\cdot}$, $\overline{R}_{\cdot,\cdot}$ and $\underline{R}_{\cdot,\cdot}$ share the same structural form. In~\cite{DHCLL22}, for $y \in \R$, these functions are defined through the following generic expression:
\begin{align*}
\left\{
    \begin{array}{ll}
        0 &\mbox{if } y \leq K-\kappa, \\
        -\frac{1}{4\kappa^3} y^3 + \frac{3K}{4\kappa^3} y^2 - \frac{3K^2-3\kappa^2}{4\kappa^3}y+\frac{K^3+2\kappa^3-3 K \kappa^2}{4\kappa^3} &\mbox{if } K -\kappa \leq y \leq K + \kappa,\\
        1 &\mbox{if } y \geq K + \kappa,
    \end{array}
\right.
\end{align*}
where $\kappa$ and $K$ are positive constants with $K-\kappa>0$. With this precise definition, the functions are smooths. Instead, we will generalize this assumption, see~{\bf (H3)} below. Finally, for later uses we also introduce the sets
\begin{align}
\Esc := \{ {\rm s}\}, \quad \Epc :=\{{\rm p, s}\}, \quad \Eent :=\{{\rm e,p}\}, \quad \Egc   :=\{{\rm g, p}\},
\end{align}
and
\begin{align}\label{def : f}
f(x,\rho,\rhosc,\rhopc,\rhoent,\rhogc,c_{\rm b}) &:= \sum_{i \in \T} f_i(x,\rho,(\rho_j)_{j\in\mathcal{E}_i},c_{\rm b})\\
&=\rhosc \, q_{{\rm div,s}} \, \big(1-R_{{\rm div,s}}(x)\big)\,\big(1-\overline{R}_{{\rm div,s}}(\rho)\big)\big(1-\underline{R}_{{\rm div,s}}(c_{\rm b})\big)\nonumber\\
&+\rhopc\, q_{{\rm div,p}}\, \big(1-R_{{\rm div,p}}(x)\big) \, \big(1-\overline{R}_{{\rm div,p}}(\rho)\big)\nonumber\\
&- \rhoent \, q_{{\rm ex,e}} \, R_{{\rm ex,e}}(x) \, \overline{R}_{{\rm ex,e}}(\rho)- \rhogc \, q_{\rm ex, g} \, R_{\rm ex, g}(x) \, \overline{R}_{\rm ex, g}(\rho).\nonumber
\end{align}

The source terms model different events for each type of cells, namely symmetric division, differentiation or death by extrusion. For instance, the first term in~\eqref{def : fsc} models the symmetric division of stem cells while the second models the differentiation of stem cells into progenitor cells. Moreover, in~\eqref{def : fsc} and thanks to the assumptions on the functions $R_{{\rm div, sc}}$, $\overline{R}_{{\rm div, sc}}$ and $\underline{R}_{{\rm div, sc}}$, we observe that the symmetric division is regulated by three different factors: the position in the crypt, the total density of cells and the concentration of butyrate. We refer the interested reader to~\cite[Section 2.3]{DHCLL22} for a thorough discussion concerning the biological interpretations of the definitions~\eqref{def : fsc}--\eqref{def : fgc}.


\subsection{Mathematical difficulties and related studies}\label{sec : difficulties}

If the coupling between the cross-diffusion part and the diffusion-reaction part of the system~\eqref{eq : rhosc}--\eqref{def : fgc} needs to be treated with care, the major mathematical issues for the analysis of this model arise from its cross-diffusion part. Indeed, it is nowadays well known that the analysis of cross-diffusion systems is challenging~\cite{Jue15}. In particular, the nonlinearities as well as the strong coupling between the equations of these systems make the establishment of convenient estimates on the unknown functions delicate. Moreover, in the present system, the lack of linear diffusion terms on the densities makes the derivation of Sobolev type estimates on $\rho_i$ even more delicate. In our subsequent analysis, we will show that the $\rho_i$ are BV functions in space.

In fact, it has been shown in several papers that cross-diffusion systems of the form~\eqref{eq : rhosc}--\eqref{eq : rhogc} exhibit, in one or several space dimensions, a segregation phenomenon, see~\cite{BGHP85,BGH871,BGH872}, for the case without source terms and~\cite{BPM10,BHIM12,BHIMW20} for the case with Lotka-Volterra source terms. This means that if initially the densities have compact and disjoint supports, then this property is preserved along time. In particular, the densities are discontinuous functions in space, i.e., these functions are merely BV in space. The segregation property is mainly due to the parabolic-hyperbolic structure of the equations~\eqref{eq : rhosc}--\eqref{eq : rhogc}. Indeed, neglecting the source terms to simplify the presentation, it is clear that $\rho$ solves a degenerate parabolic equation, while $\rho_i$ is a solution to a conservative transport equation with velocity $-\pa_x \rho$ for any $i \in \T$. 

However, here the situation is a bit different and it has been illustrated numerically in~\cite{DHCLL22,darri2020}, that such segregation phenomenon does not seem to occur for the system~\eqref{eq : rhosc}--\eqref{def : fgc}. This is mainly due to the definitions of the source terms in the model. Indeed, if it was shown previously in the literature that such phenomena occur in cross-diffusion systems of the form~\eqref{eq : rhosc}--\eqref{eq : rhogc} with reaction terms, the authors always (to our best knowledge) considered Lotka-Volterra type reaction (or at least no cross-reaction terms, see~\cite{CFSS_JKO_18}). In particular, for these type of reaction terms the coupling term is usually quadratic, i.e., the product of two different densities. However, in the definitions~\eqref{def : fsc}--\eqref{def : fgc}, some coupling terms depend linearly on the densities. For instance, the differentiation of stem cells into progenitor cells is modelled in~\eqref{def : fsc} and~\eqref{def : fpc} through $\pm \rhosc \, q_{{\rm s,p}} \, R_{\rm s, p}(x)$ and similarly for the other reaction terms. Therefore, it is expected to observe at each colony interfaces a mixture of different type of cells. The study of the regularity of these interfaces will be the subject of future research studies. However, in this first work, as for instance in~\cite{BPM10,BGH872,BGHP85,BGH871,BHIMW20,GalSel15 }, we will only look for partial densities $\rho_i$ admitting BV regularity.

The analysis performed in~\cite{BPM10,BGH872,BGHP85,BGH871,BHIMW20}, in dimension one, are based on change of variables together with a fine analysis of the associated free boundary problems. In~\cite{BHIM12}, the authors also used the theory of renormalized solutions introduced by Di Perna and Lions~\cite{DL89} to compensate the lack of suitable BV estimates in higher dimension. If these techniques are elegant and well adapted for the theoretical analysis of systems of the form~\eqref{eq : rhosc}--\eqref{eq : rhogc}, it seems difficult to adapt these methods at the discrete level to prove the convergence of a numerical scheme. Instead, we will adapt in our context the vanishing viscosity method developed in~\cite{GalSel15}.

Before to conclude this section, let us give a wider and non exhaustive overview of the analysis of parabolic-hyperbolic cross-diffusion systems. In particular, we refer to the series of papers~\cite{DrHoJu23,DrJu20,HoJu25}, where the authors study and give a detailed state of the art concerning the analysis of such systems. We also mention~\cite{CFSS_JKO_18}, where existence and segregation results are obtained in one space dimension thanks to an optimal transport approach. If in~\eqref{eq : rhosc}--\eqref{eq : rhogc}, we only consider a pressure function $p$ given by $p(\rho)=-\pa_x \rho$, in~\cite{CFSS_JKO_18} (see also~\cite{BPM10}), the authors consider the more general case $p(\rho)=-\pa_x \chi'(\rho)$ where the function $\chi$ described an internal energy density satisfying convenient regularity assumptions. Similarly, we mention~\cite{PQV14,BPPS20,PerVau15} and references therein, where the authors study systems of the form~\eqref{eq : rhosc}--\eqref{eq : rhogc} (with different source terms) and where $p(\rho) \approx \rho^{m-1}$ with $m>1$. These systems arise in the modelling of cancer development. In particular, in these papers the authors study the limit $m \to \infty$ and derive rigorously free boundary models of Hele-Shaw type.


\subsection{Assumptions and main result}\label{sec : main results}

In this section, we present our main result and explain the main ideas of its proof. Let us first introduce some functional spaces. In particular, we define the space $BV(\Omega)$ as
$$
BV(\Omega)=\left\{g \in L^{1}(\Omega) \, : \, T V(g)<+\infty\right\},
$$
where $TV(g)$ denotes the total variation of $g$ in $\Omega$ given by:
\begin{align}\label{def TV}
TV(g)=\sup \left\{ \int_{\Omega} g(x)\, \varphi'(x) \, dx \, : \, 
\varphi \in C_c^1(\Omega), \, \|\varphi\|_\infty \le 1 \right\}.
\end{align}
We also introduce the space
\begin{align}\label{def : H}
H :=\{ c \in H^1(\Omega) \, : \, c(1) = 0 \},
\end{align}
endow with the $H^1(\Omega)$ norm. Now, let us gather our main assumptions:
\begin{itemize}
\item[\textbf{(H1)}] Initial data for the densities: For any $i \in \T$, the function $\rho^0_i$ is nonnegative and bounded such that $\rho^0 = \sum_{i \in \T} \rho^0_i \in H^1(\Omega)$. Moreover, we assume that 
\begin{equation}
    w^0_i(x) \coloneq \frac{\rho_i^0}{\rho^0}(x) \in~BV(\Omega).
    \end{equation}
\item[\textbf{(H2)}] Initial data for the concentration: The nonnegative function $c_{\rm b}^0$ belongs to $H^1(\Omega)$.

\item[\textbf{(H3)}] Parameters: The constants $q_{.,.}$ appearing in the definitions~\eqref{def : fsc}--\eqref{def : fgc} are positive as well as $\sigma_{\rm b}$, $\gamma$ and $c_{\rm b}^d$.

\item[\textbf{(H4)}] Reaction terms: The functions $R_{\cdot,\cdot}$, $\overline{R}_{\cdot,\cdot}$ and $\underline{R}_{\cdot,\cdot}$ appearing in the definitions~\eqref{def : fsc}--\eqref{def : fgc} are $W^{1,\infty}$ non-decreasing functions with values in $[0,1]$ and equal to $0$ on $\R_-$. Moreover, there exists a common $\overline{M}>0$ such that $\overline{R}_{{\rm div,s}}$, $\overline{R}_{{\rm div,p}}$, $\overline{R}_{{\rm ex,e}}$ and $\overline{R}_{{\rm ex,g}}$ are equal to $1$ on $(\overline{M},+\infty)$. 
\end{itemize}

\begin{remark}
In contrast with some existing results in the literature dealing with the existence of solutions to systems of the type~\eqref{eq : rhosc}--\eqref{eq : rhogc}, see for instance~\cite{GalSel15}, we do not assume that $\rho^0$ is positive. This is biologically meaningful in our context, since initially only the bottom of the crypt is occupied by stem cells.
\end{remark}

Then, the main objective of this paper is to prove the following:

\begin{theorem}[Existence of weak solutions]\label{thm1}
Let the assumptions~\emph{{\bf (H1)}--{\bf (H4)}} hold. Then, there exist nonnegative functions $\rhosc$, $\rhopc$, $\rhoent$ and $\rhogc$ such that $\rho_i \in L^\infty(0,T;BV(\Omega))$ with $\pa_t \rho_i \in L^2(0,T;(H^2)'(\Omega))$ for any $i \in \T$ and $\rho\in L^2(0,T;H^1(\Omega))$. Moreover, there exists a nonnegative function $c_{\rm b} \in L^2(0,T;H)$ with $\pa_t c_{\rm b} \in L^2(0,T;H')$ where $H'$ denotes the dual space of $H$. Furthermore, for any $\varphi \in L^2(0,T;H^1(\Omega))$ and $i \in \T$, it holds
\begin{align} \label{ eq rhoi weak final}
\int_{\Omega \times (0,T)} \pa_t \rho_i \, \varphi \, \dd x \dd t + \int_{\Omega \times (0,T)} \rho_i \, \pa_x \rho \, \pa_x \varphi \, \dd x \dd t= \int_{\Omega \times (0,T)} f_i(x,\rho,(\rho_j)_{j \in \mathcal{E}_i},c_{\rm b}) \, \varphi \, \dd x \dd t,
\end{align}
and for all $\psi \in L^2(0,T;H)$
\begin{multline}\label{ eq cb weak final }
\int_{\Omega \times (0,T)} \pa_t c_{\rm b} \, \psi \, \dd x \dd t +\sigma_{\rm b} \int_{\Omega \times (0,T)} \pa_x c_{\rm b} \, \pa_x \psi \, \dd x \dd t\\= \gamma\int_{\Omega \times (0,T)} \dfrac{c_{\rm b}+c_{\rm b}^d}{1+c_{\rm b}+c_{\rm b}^d}\, (\rhoent+\rhogc) \, \psi \, \dd x \dd t,
\end{multline}
where we recall definition~\eqref{def : H} of the space $H$.
\end{theorem}

The proof of this result relies on a vanishing viscosity method. We will first study, in Section~\ref{subsec : regul}, the existence of (weak) solutions to a regularized system associated to \eqref{eq : rhosc}--\eqref{def : fgc} thanks to a semi-discrete in time scheme. Then, in Section~\ref{sec : exi proof}, we will establish uniform, with respect to the viscosity parameter, estimates. Finally, we will show that these estimates are sufficient to obtain compactness properties and to pass to the limit on this parameter. This will lead to the existence of at least one weak solution to \eqref{eq : rhosc}--\eqref{def : fgc}.


\section{Study of a regularized system}\label{subsec : regul}

As already explained, the proof of Theorem~\ref{thm1} relies on a vanishing viscosity method. Therefore, we first introduce a regularized system associated to \eqref{eq : rhosc}--\eqref{def : fgc}. Then, let $\eps >0$, this system writes:
\begin{align}\label{eq : rhosc eps}
\pa_t \rhose - \eps \pa_x^2 \rhose - \pa_x(\rhosc^{\eps} \, \pa_x \rho^\eps) &= \fsc\left(x,\rho^\eps-\eps,\rhose, c^\eps_{\rm b}\right), \,\,\,\,\qquad \mbox{in }\Omega \times (0,T),\\
\pa_t \rhope - \eps \pa_x^2 \rhope - \pa_x(\rhope \,\pa_x \rho^\eps) &= \fpc\left(x,\rho^\eps-\eps,\rhope,\rhose, c^\eps_{\rm b}\right), \quad \mbox{in }\Omega \times (0,T),\label{eq : rhopc eps}\\
\pa_t \rhoee - \eps \pa_x^2 \rhoee - \pa_x(\rhoee \,\pa_x \rho^\eps) &= \fent\left(x,\rho^\eps-\eps,\rhoee,\rhope, c^\eps_{\rm b}\right), \quad \mbox{in }\Omega \times (0,T),\label{eq : rhoent eps}\\
\pa_t \rhoge - \eps \pa_x^2 \rhoge - \pa_x(\rhoge \, \pa_x \rho^\eps) &= \fsc\left(x,\rho^\eps-\eps,\rhoge,\rhope, c^\eps_{\rm b}\right), \quad \mbox{in }\Omega \times (0,T),\label{eq : rhogc eps}\\
\pa_t c^\eps_{\rm b} - \sigma_{\rm b} \pa_x^2 c^\eps_{\rm b} &= \gamma \dfrac{c^\eps_{\rm b}+c_{\rm b}^d}{1+c^\eps_{\rm b}+c_{\rm b}^d} \, (\rhoee + \rhoge), \,\,\quad \mbox{in } \Omega \times (0,T),
\end{align}
with $\rho^\eps := \sum_{i \in \T} \rho_i^\eps$. We complement this system with the following boundary and initial conditions:
\begin{align}
\pa_x \rho^\eps_i(0,t) &= \pa_x \rho^\eps_i(1,t) = 0 \quad \mbox{for } t\in (0,T) \mbox{ and } i \in \T,\label{boundary rho eps}\\
\pa_x c^\eps_{\rm b}(0,t) &= 0, \quad c^\eps_{\rm b}(1,t) = 0 \quad \mbox{for } t \in (0,T),\label{boundary c eps}\\
\rho^\eps_i(x,0) &= \rho^{\eps,0}_i(x) \quad \mbox{for } x \in \Omega \mbox{ and }  i \in \T,\\
c^\eps_{\rm b}(x,0) &= c_{\rm b}^0(x) \quad \mbox{for } x \in \Omega,\label{def : IC cbeps}
\end{align}
where we define 
\[
\rho_i^{\eps,0}(x)=\rho_i^0(x)+\frac\eps4,\text{ for all }x\in\Omega\text{ and } i\in \T,
\]
so that, $\rho^{\eps,0}(x)=\rho^{0}(x)+\,\eps$. Finally, in~\eqref{eq : rhosc eps}--\eqref{eq : rhogc eps} the definitions of the source terms are given by~\eqref{def : fsc}--\eqref{def : fgc}. The purpose of this regularization is to obtain a nondegenerate system.

\begin{proposition}\label{prop : semi discrete}
Let assumptions \emph{{\bf (H1)}--{\bf (H4)}} hold and assume that $0<\eps<1$. Then, for all $i \in \T$, there exist nonnegative functions $\rho_i^\eps \in L^2(0,T;H^1(\Omega))$ with $\pa_t \rho_i^\eps \in L^2(0,T;(H^1)'(\Omega))$ and $c_{\rm b}^\eps \in L^2(0,T;H)$ with $\pa_t c_{\rm b}^\eps \in L^2(0,T;H')$ where we recall that $H'$ denote the dual space of $H$, see~\eqref{def : H}. Moreover, for any $\varphi \in L^2(0,T;H^1(\Omega))$ and $i \in \T$, it holds
\begin{multline}\label{weak eq rhoi eps}
\int_{\Omega \times (0,T)} \pa_t \rho^\eps_i \, \varphi \, \dd x \dd t + \eps \int_{\Omega \times (0,T)} \pa_x \rho_i^\eps \, \pa_x \varphi \, \dd x\dd t + \int_{\Omega \times (0,T)} \rho^\eps_i \, \pa_x \rho^\eps \, \pa_x \varphi \, \dd x \dd t \\ 
= \int_{\Omega \times (0,T)} f_i(x,\rho^\eps-\eps,(\rho^\eps_j)_{j \in \mathcal{E}_i},c^\eps_{\rm b}) \, \varphi \, \dd x \dd t,
\end{multline}
and for all $\psi \in L^2(0,T;H)$
\begin{multline}\label{weak eq cb eps}
\int_{\Omega \times (0,T)} \pa_t c^\eps_{\rm b} \, \psi \, \dd x \dd t +\sigma_{\rm b} \int_{\Omega \times (0,T)} \pa_x c^\eps_{\rm b} \, \pa_x \psi \, \dd x \dd t\\= \gamma\int_{\Omega \times (0,T)} \dfrac{c^\eps_{\rm b}+c_{\rm b}^d}{1+c^\eps_{\rm b}+c_{\rm b}^d} (\rhoee+\rhoge) \, \psi \, \dd x \dd t.
\end{multline}
\end{proposition}

In order to prove this result, we will use a semi-discrete in time scheme. More precisely, for a given positive integer $N_T$, we define a time step $\Delta t := T/N_T$ and a sequence $t_n = n \Delta t$ for any $n=0,\ldots,N_T$. Then, for some given $n \in\{0,\ldots,N_T-1\}$, starting from given nonnegative functions $(\rho^{\eps,n}_i)_{i \in \T}$  and   $c_{\rm b}^{\eps,n}$ such that  $\rho^{\eps,n}_i, c_{\rm b}^{\eps,n}\in L^2(\Omega)$  for all $i \in \T$ with the condition  \begin{equation}\label{star}
    \eps\leq\rho^{\eps,n}\leq M_\infty^\eps\text{ a.e in $\Omega$} \tag{$\star$},
\end{equation}
where $M_\infty^\eps$ is defined as
\begin{align}\label{def : Minfty}
M^\eps_\infty \coloneq M_\infty+\eps= \max\left(\|\rho^{0}\|_{L^\infty(\Omega)}, \, \overline{M}\right)+\eps,
\end{align}
in which the constant $\overline{M}>0$ is introduced in  \emph{{\bf (H4)}}, we first solve:
\begin{multline}\label{eq : rho eps delta}
\rho^{\eps,n+1} - \eps \,\Delta t \, \pa_x^2 \rho^{\eps,n+1} -\Delta t\, \pa_x(\rho^{\eps,n} \, \pa_x \rho^{\eps,n+1}) \\ = \rho^{\eps,n} + \Delta t\, f\left(x,\rho^{\eps,n+1}-\eps,\rhosc^{\eps,n},\rhopc^{\eps,n},\rhoent^{\eps,n},\rhogc^{\eps,n}, c^{\eps,n}_{\rm b}\right), \quad \mbox{in }\Omega,
\end{multline}
where we recall definition~\eqref{def : f} of $f$. Then, for all $i \in \T$, we define the functions $\rho^{\eps,n+1}_i$ as solution to the following system
\begin{multline}\label{eq : rhoi eps delta}
\rho_i^{\eps,n+1} - \eps \, \Delta t\, \pa_x^2 \rho_i^{\eps,n+1} -\Delta t \, \pa_x(\rho_i^{\eps,n+1} \,\pa_x \rho^{\eps,n+1}) \\
= \rho^{\eps,n}_i + \Delta t \, f_i\left(x,\rho^{\eps,n+1}-\eps,(\rho^{\eps,n+1}_j)_{j \in \mathcal{E}_i}, c^{\eps,n}_{\rm b}\right), \quad \mbox{in }\Omega.
\end{multline}
Finally, we define $c_{\rm b}^{\eps,n+1}$ as solution to
\begin{align}\label{eq : cb eps delta}
c^{\eps,n+1}_{\rm b} - \sigma_{\rm b} \, \Delta t \, \pa_x^2 c^{\eps,n+1}_{\rm b} &= c^{\eps,n}_{\rm b}+ \gamma \, \Delta t \, \dfrac{c^{\eps,n}_{\rm b}+c_{\rm b}^d}{1+c^{\eps,n}_{\rm b}+c_{\rm b}^d} \, (\rhoent^{\eps,n+1} + \rhogc^{\eps,n+1}), \quad \mbox{in } \Omega.
\end{align}
Of course, we complement system~\eqref{eq : rho eps delta}--\eqref{eq : cb eps delta} with homogeneous Neumann boundary conditions for $\rho^{\eps,n+1}$ and $\rho^{\eps,n+1}_i$, for all $i \in \T$ as in~\eqref{boundary rho eps}, and mixed boundary conditions for $c_{\rm b}^{\eps,n+1}$ as in~\eqref{boundary c eps}. Throughout this section, to simplify the notation, we will neglect the exponent $\eps$ and simply write, for instance, $\rho^{n+1}$ instead of $\rho^{\eps,n+1}$.

Thanks to the choice of the discretization in time of the equations~\eqref{eq : rho eps delta}--\eqref{eq : cb eps delta}, we notice that these equations are decoupled. In particular, we will first study the equation on the total density~\eqref{eq : rho eps delta}.
\begin{remark}
Of course, instead of considering all four functions $(\rho^{n+1}_i)_{i\in\T}$, 
the analysis can be restricted to the three functions indexed by $i\in \widetilde\T = \T \setminus \{\mathrm{g}\}$. The existence of the remaining component \( \rhogc^{n+1} \) is then obtained by $\rhogc^{n+1}=\rho^{n+1}-\sum_{i\in\widetilde\T}\rho_i^{n+1}$ and the nonnegativity of $\rho^{n+1}_{\rm g}$ is shown as for the other species.
\end{remark}


\subsection{Study of the equation on the total density}

\begin{lemma}\label{lemma 1}
Let the  assumptions \emph{{\bf (H3)}-{\bf (H4)}} hold and assume that $0<\eps<1$. Moreover, for all $ i\in \T$, we assume that $\rho^n_i$ and $c_{\rm b}^n$ are nonnegative and belong to $L^2(\Omega)$ such that $\rho^{n}$ satisfies~\eqref{star}. Then, there exists $\rho^{n+1} \in H^1(\Omega)$ weak solution to~\eqref{eq : rho eps delta} satisfying $\eps \leq ~\rho^{n+1}(x) \leq ~M_\infty^\eps,$ for all $x \in \Omega,$ and  the following estimate 
\begin{equation}\label{coercivity-L2 rho unif delta}
   \| \rho^{n+1}\|^2_{L^2(\Omega)}+2\eps\,\Delta t \,\| \pa_x\rho^{n+1}\|^2_{L^2(\Omega)}\leq \| \rho^n\|^2_{L^2(\Omega)}+8\Delta t\,q_\infty   {(M_\infty^\eps)}^2(1+2q_\infty\,\Delta t),
\end{equation}
where $q_\infty\coloneq\max(q_{{\rm div,s}},\ldots, q_{\rm ex, g})$.
\end{lemma}
\begin{proof}
We first linearized~\eqref{eq : rho eps delta}. Indeed, for a given $v\in L^2(\Omega)$, we look for $\rho\in H^1(\Omega)$ such that  \begin{equation} \label{weak formulation 1}
    a(\rho,\varphi)=L(\varphi),\quad \forall\varphi\in H^1(\Omega),
\end{equation}
where $a(\cdot,\cdot)$ is a bilinear form defined on $H^1(\Omega)\times H^1(\Omega)$ as follows:
\begin{equation*}
a(\rho,\varphi)=\int_{\Omega} \rho \, \varphi\,\dd x+\eps \,\Delta t\,\int_{\Omega} \pa_x\rho \, \pa_x\varphi\,\dd x+\Delta t \int_{\Omega} \rho^n \, \pa_x\rho\,\pa_x  \, \varphi\,\dd x,
 \end{equation*}
and the linear form $L(\cdot)$ defined on $H^1(\Omega)$ as:
\begin{equation*}
L(\varphi)=\int_{\Omega} \rho^n\varphi\,\dd x+\Delta t \int_{\Omega} f\left(x,v-\eps,\rhosc^{n},\rhopc^{n},\rhoent^{n},\rhogc^{n}, c^{n}_{\rm b}\right)\varphi\,\dd x.
\end{equation*}
Then, we split the proof in three steps:\medskip

\noindent\textit{\textbf{Step 1. Well posedness of \eqref{weak formulation 1}.}}
It is clear that $L(.)$ and $a(.,.)$ are respectively continuous in $H^1(\Omega)$ and $(H^1(\Omega))^2$ by using  \eqref{star} and the assumtions \emph{{\bf (H4)}}.  Moreover, for all $\varphi\in H^1(\Omega)$ we have 
\begin{align*}
a(\varphi,\varphi)&=\int_{\Omega} \varphi^2\ \dd x +\Delta t\int_{\Omega}(\eps+\rho^n)\,|\pa_x\varphi|^2\dd x\geq\min(1,\eps\,\Delta t)\|\varphi\|^2_{H^1(\Omega)}.
\end{align*}
Therefore, $a(\cdot,\cdot)$ is coercive and by Lax-Milgram theorem there exists a unique solution $\rho\in~H^1(\Omega)$ to \eqref{weak formulation 1}.
Moreover, using $\rho$ as a test function in \eqref{weak formulation 1} and the formula $ab\leq\frac{a^2}{2}+\frac{b^2}{2}$, we obtain
\begin{equation*}
    \| \rho\|^2_{L^2(\Omega)}+\eps\,\Delta t\, \| \pa_x\rho\|^2_{L^2(\Omega)}+\Delta t \int_\Omega\rho^n|\pa_x\rho|^2\dd x \leq \frac12\int_\Omega\left(\rho^n+\Delta t\,f\right)^2\dd x+\frac12 \| \rho\|^2_{L^2(\Omega)}.
\end{equation*}
Recalling the expression \eqref{def : f} of $f$  and using \eqref{star} together with \emph{{\bf (H4)}} we get 
\begin{equation}
\|f\|_{L^\infty(\Omega)}\leq 4 M_\infty^\eps q_\infty.\label{star star} \tag{$\star\star$}
\end{equation} 
Then, we obtain \eqref{coercivity-L2 rho unif delta} with $\rho$ instead of $\rho^{n+1}$.

\noindent\textit{\textbf{Step 2. Schaefer's fixed point theorem.}}
In order to apply Schaefer's fixed point theorem \cite[Theorem 4, Chapter 9]{MR2597943}, we introduce the following map : 
\begin{align*}
 \mathcal{F} :  L^2(\Omega) &\longrightarrow L^2(\Omega),\\ 
 v &\longmapsto \rho,
\end{align*}
where $\rho=\mathcal{F}(v)\in H^1(\Omega)$ is the unique solution of the weak formulation \eqref{weak formulation 1}.\\
First to prove the continuity of $\mathcal{F}$, we consider, for $k\in\N$, the sequence $(v_k)_k\in L^2(\Omega)$ that converges strongly toward $v$ in $ L^2(\Omega)$. Thanks to~\eqref{coercivity-L2 rho unif delta} (with $\rho_k$ instead of $\rho^{n+1}$), we deduce that $\rho_k=\mathcal{F}(v_k)$ is bounded in $H^1(\Omega)$. Then, up to a subsequence, as $k\to+\infty$, we have that  
\begin{align*}
    \pa_x\rho_k&\rightharpoonup \pa_x\rho\quad\text{ weakly in } H^1(\Omega),\\
    \rho_k &\longrightarrow \rho \quad\text{  strongly in } L^2(\Omega).
\end{align*}
Moreover, thanks to the assumption  \emph{{\bf (H4)}}, the function $f$ is continuous w.r.t.~each component. Therefore, taking the limit $k\to+\infty,$  and using Lebesgue dominated convergence theorem, we get $a(\rho,\varphi)=L(\varphi)$, for all $\varphi~\in~H^1(\Omega)$ and $\mathcal{F}$ is continuous.
Furthermore, since for a bounded sequence $(v_k)_k$ in $L^2(\Omega)$, $\mathcal{F}(v_k)=\rho_k$ converges, up to a subsequence, strongly in $L^2(\Omega)$. Then, $\mathcal{F}$ is compact.
Finally, we consider $v\in L^2(\Omega)$ as a solution to $$v=\lambda\mathcal{F}(v)\qquad \text{ for } \lambda\in [0,1].$$
Using (\ref{coercivity-L2 rho unif delta}) we have that \begin{align*}
    \|\mathcal{F}(v)\|_{H^1(\Omega)}\leq C.
\end{align*}
Therefore , $v=\lambda\mathcal{F}(v)$ is bounded in $H^{1}(\Omega)$ which yields the existence of a fixed point  denoted $\rho^{n+1}$ of $\mathcal{F}$.
Hence, $\rho^{n+1}$ is a  weak solution to \eqref{eq : rho eps delta}, i.e.~for all $\varphi\in H^1(\Omega),$ it holds 
\begin{multline} \label{weak rho}
    \int_{\Omega} \rho^{n+1}\varphi\,\dd x+\eps\,\Delta t \int_{\Omega} \pa_x\rho^{n+1}\pa_x\varphi\,\dd x+\Delta t \int_{\Omega} \rho^n \pa_x\rho^{n+1}\pa_x\varphi\,\dd x\\=\int_{\Omega} \rho^n\varphi\,\dd x+\Delta t \int_{\Omega} f\left(x,\rho^{n+1}-\eps,\rhosc^{n},\rhopc^{n},\rhoent^{n},\rhogc^{n}, c^n_{\rm b}\right)\varphi\,\dd x.
\end{multline}

\noindent\textit{\textbf{Step 3. Positivity and Boundedness of $\rho^{n+1}$.}}
Using $[\rho^{n+1}-\eps]_-$ as a test function in \eqref{weak rho}, where $[x]_-=\min(x,0)$, we have 
\begin{multline*}
    \int_{\Omega} \left(\rho^{n+1}-\rho^{n}\right)[\rho^{n+1}-\eps]_-\dd x+\Delta t \int_{\Omega} (\eps+\rho^n) |\pa_x[\rho^{n+1}-\eps]_-|^2\dd x\\=\Delta t \int_{\Omega}f\left(x,\rho^{n+1}-\eps,\rhosc^{n},\rhopc^{n},\rhoent^{n},\rhogc^{n}, c^{n}_{\rm b}\right)[\rho^{n+1}-\eps]_-\dd x.
\end{multline*}
Thanks to the assumptions \emph{{\bf (H3)--\bf (H4)}}, the definition ~\eqref{def : f} of the source term and the nonnegativity property  of $(\rho_i^{n})_{i\in\T}$, we deduce that the second term in the left hand side~of the above equation is nonnegative as well as the term in the right hand side. Then, we get
\begin{equation*}
    \int_{\Omega} \left((\rho^{n+1}-\eps)-(\rho^n-\eps)\right)[\rho^{n+1}-\eps]_-\dd x\leq0.
\end{equation*}
Since $\rho^n\geq\eps$ for a.e.~in $\Omega$, we obtain 
\begin{align*}
     \int_{\Omega} |[\rho^{n+1}-\eps]_-|^2\dd x\leq 0,
\end{align*}
so that $\rho^{n+1}\geq\eps$ for a.e. in $\Omega$.

Now, in order to prove the boundedness of $\rho^{n+1}$, we use  $[\rho^{n+1}-M_\infty^\eps]_+ $ as a test function in~\eqref{weak rho}, with $[x]_+\coloneq\max(x,0)$, and where we recall definition~\eqref{def : Minfty} of $M_\infty^\eps$. Then, we have 
\begin{multline*}
    \int_{\Omega}\left( \rho^{n+1}-\rho^n\right)[\rho^{n+1}-M_\infty^\eps]_+\dd x+\Delta t \int_{\Omega} (\eps+\rho^n) |\pa_x[\rho^{n+1}-M_\infty^\eps]_+|^2\dd x\\=\Delta t \int_{\Omega}f\left(x,\rho^{n+1}-\eps,\rhosc^{n},\rhopc^{n},\rhoent^{n},\rhogc^{n}, c^{n}_{\rm b}\right)[\rho^{n+1}-M_\infty^\eps]_+\dd x.
\end{multline*}
Using assumption \emph{\bf(H4)} and the nonnegativity of $(\rho_i^{n})_{i\in\T}$, we obtain  
\begin{align*}
\int_{\Omega} |[\rho^{n+1}-M^\eps_\infty]_+|^2 \, \dd x\leq 0.
\end{align*}
Therefore, we deduce that $\eps\leq\rho^{n+1}\leq M^\eps_\infty$. This concludes the proof of Lemma \ref{lemma 1}.
\end{proof}


\subsection{Study of the equations on the partial densities}

\begin{lemma}\label{lemma 2}
Let the assumptions of Lemma \ref{lemma 1} hold. Then, under the condition $q_\infty\,\Delta t<1,$ there exist nonnegative functions $\rhosc^{n+1}$, $\rhopc^{n+1}$, $\rhoent^{n+1}$ and $\rhogc^{n+1}$ belonging to $H^1(\Omega)$, weak solutions to~\eqref{eq : rhoi eps delta} satisfying, for all $i\in\T$,
\begin{multline}
\label{L2 rhoi unif deltat}
   \| \rho_i^{n+1}\|^2_{L^2(\Omega)}+\eps\, \Delta t\, \| \pa_x\rho_i^{n+1}\|^2_{L^2(\Omega)}\leq \| \rho_i^n\|^2_{L^2(\Omega)}\\+\frac{\Delta t}\eps  (M_\infty^\eps)^2\,\|\pa_x\rho^{n+1}\|^2_{L^2(\Omega)}+8q_\infty\,\Delta t \, {(M_\infty^\eps)}^2(1+2 q_\infty\,\Delta t ).
\end{multline}
\end{lemma}

\begin{proof}
In order to prove this result, let us observe, thanks to the definitions~\eqref{def : fsc}--\eqref{def : fgc}, that the equations of system~\eqref{eq : rhoi eps delta} are decoupled. Indeed, one can first solve the equation on $\rhosc^{n+1}$, then on $\rhopc^{n+1}$ and finally on $\rhoent^{n+1}$. Since all these equations are quite similar, we only give the detailed analysis for $\rhosc^{n+1}$. First, we introduce the truncation operator defined on $\R$ as
 \begin{equation*} 
    T^{M^\eps_\infty}(x)=\begin{cases}
        0  \quad&\text{ if } x\leq 0,\\x \quad&\text{ if }0\leq x\leq M^\eps_\infty,
        \\
    M^\eps_\infty &\text{ if }x\geq M^\eps_\infty.
    \end{cases}
\end{equation*}
In order to prove Lemma~\ref{lemma 2}, we linearized \eqref{eq : rhoi eps delta} as follows: for a given $\vsc\in L^2(\Omega)$, we look for a solution $\rhosc\in H^1(\Omega)$ to
\begin{equation} \label{weak formulation 2}
    a_{\rm s}(\rhosc,\varphi)=L_{\rm s}(\varphi), \quad\forall\varphi\in H^1(\Omega),
\end{equation}
where $ a_{\rm s}(\cdot,\cdot)$ is a bilinear form defined on $H^1(\Omega)\times H^1(\Omega)$ by:
\begin{equation*}  a_{\rm s}(\rhosc,\varphi)=\int_{\Omega}\rhosc\,\varphi\dd x + \eps\, \Delta t\,\int_{\Omega}\pa_x\rhosc\,\pa_x\varphi\dd x-\Delta t\int_{\Omega}\fsc\left(x,\rho^{n+1}-\eps,T^{M_\infty^\eps}(\rhosc), c^{n}_{\rm b}\right)\,\varphi\,\dd x,
\end{equation*}
and the linear form $L_s(\cdot)$ defined on $H^1(\Omega)$ as:
\begin{equation*}
     L_{\rm s}(\varphi)=\int_{\Omega}\rhosc^{n}\,\varphi\dd x-\Delta t \int_{\Omega}T^{M_\infty^\eps}(\vsc)\,\pa_x\rho^{n+1}\,\pa_x\varphi\,\dd x.
\end{equation*}
As in the proof of the previous lemma, we split the proof in two steps.\medskip

\noindent\textit{\textbf{Step 1. Existence of a fixed point.}} It is clear that $ a_{\rm s}(.,.)$ and $L_s(.)$ are respectively continuous on $(H^1(\Omega))^2$ and $H^1(\Omega)$. Moreover, under assumption \textbf{(H4)} and by recalling definition \eqref{def : fsc} of the source term $f_s$, it follows that for every $\varphi \in H^1(\Omega)$, we have 
\begin{align*}
 a_{\rm s}(\varphi,\varphi)&=\int_{\Omega} \varphi^2\ \dd x +\eps\, \Delta t\,\int_{\Omega}|\pa_x\varphi|^2\dd x+\Delta t\int_\Omega T^{M_\infty^\eps}(\varphi)\varphi \, q_{{\rm s,p}} \, R_{\rm s, p}(x)\dd x\\&-\Delta t \int_{\Omega} T^{M_\infty^\eps}(\varphi)\varphi q_{{\rm div,s}} \, \big(1-R_{{\rm div,s}}(x)\big)\,\big(1-\overline{R}_{{\rm div,s}}(\rho^ {n+1}-\eps)\big)\big(1-\underline{R}_{{\rm div,s}}(c^n_{\rm b})\big)\ \dd x \\&\geq(1-\Delta tq_{{\rm div,s}})\|\varphi\|^2_{L^2(\Omega)}+\eps\, \Delta t\,\|\pa_x\varphi\|^2_{L^2(\Omega)}\geq \min(1-\Delta t\,q_{{\rm div,s}},\eps\,\Delta t)\|\varphi\|^2_{H^1(\Omega)}.
\end{align*}
Therefore, thanks to the condition $q_\infty\,\Delta t<1$ we get the coercivity of $ a_{\rm s}(\cdot,\cdot)$. Then, by Lax-Milgram's theorem, there exists a unique solution $\rhosc$ of \eqref{weak formulation 2}.
Moreover, using $\rhosc$ as a test function in \eqref{weak formulation 2}, and applying Young and H\"older inequalities, we obtain 
\begin{multline*}
 \|\rhosc\|^2_{L^2(\Omega)}+\eps\,\Delta t\,\| \pa_x\rhosc\|^2_{L^2(\Omega)}\leq \frac12\int_\Omega\left(\rhosc^n+\Delta t\,\fsc\right)^2\dd x\\+\Delta t M^\eps_\infty\|\, \pa_x\rho^{n+1}\|_{L^2(\Omega)}\,\|\pa_x\rhosc\|_{L^2(\Omega)}+\frac12\| \rhosc\|^2_{L^2(\Omega)}.
\end{multline*}
Now, using the fact that $ab\leq \frac\eps2a^2+\frac{1}{2\eps}b^2$, we get
\begin{multline*}
    \frac12\| \rhosc\|^2_{L^2(\Omega)}+\frac{\eps\,\Delta t}2\| \pa_x\rhosc\|^2_{L^2(\Omega)}\leq\frac12\| \rhosc^{n}\|^2_{L^2(\Omega)}\\+\Delta t \,\|\rhosc^{n}\|_{L^\infty(\Omega)}\,\|\fsc\|_{L^\infty(\Omega)}+\frac12\,{\Delta t}^2\,\|\fsc\|_{L^\infty(\Omega)}^2+\frac{(M_\infty^\eps)^2}{2\eps}\Delta t\, \| \pa_x\rho^{n+1}\|^2_{L^2(\Omega)}.
\end{multline*}
Now, since $\|\fsc\|_{L^\infty(\Omega)}\leq2M_\infty^\eps q_\infty$, we obtain
\begin{multline}\label{coercivity2}
\| \rhosc\|^2_{L^2(\Omega)}+\eps\,\Delta t\| \pa_x\rhosc\|^2_{L^2(\Omega)}\leq\| \rhosc^{n}\|^2_{L^2(\Omega)}\\+4\Delta t q_\infty(M_\infty^\eps)^2(1+\Delta tq_\infty)+\frac{(M_\infty^\eps)^2}{\eps}\,\Delta t\, \| \pa_x\rho^{n+1}\|^2_{L^2(\Omega)}.
\end{multline}

Moreover, applying once more time Schaefer’s fixed point theorem, similarly as done in the proof of Lemma~\ref{lemma 1}, we deduce the existence of $\rhosc^{n+1}$ in $H^1(\Omega)$ satisfying, for every $\varphi\in H^1(\Omega)$, 
\begin{multline} \label{weak rhosc}
    \int_{\Omega} \rhosc^{n+1}\,\varphi\,\dd x+\eps\, \Delta t\,\int_{\Omega} \pa_x\rhosc^{n+1}\,\pa_x\varphi\,\dd x+\Delta t \int_{\Omega} T^{M^\eps_\infty}(\rhosc^{n+1})\,\pa_x\rho^{n+1}\,\pa_x\varphi\,\dd x\\=\int_{\Omega} \rhosc^n\,\varphi\,\dd x+\Delta t\int_{\Omega}\fsc\left(x,\rho^{n+1}-\eps,T^{M_\infty^\eps}(\rhosc^{n+1}), c^{n}_{\rm b}\right)\,\varphi \,\dd x.
\end{multline}
Finally, estimate~\eqref{L2 rhoi unif deltat} is a direct consequence of~\eqref{coercivity2}.\medskip

\noindent\textit{\textbf{Step 2. Nonnegativity and boundedness of $\rhosc^{n+1}.$}}
Using $[\rhosc^{n+1}]_-$ as a test function in~\eqref{weak rhosc} and recalling definition~\eqref{def : fsc} of $f_{\rm s}$, we have 
\begin{multline*}
    \int_\Omega(\rhosc^{n+1}-\rhosc^{n})[\rhosc^{n+1}]_-\dd x+\eps\,\Delta t\int_\Omega|\pa_x[\rhosc^{n+1}]_-|^2\dd x=-\Delta t\int_\Omega T^{M^\eps_\infty}(\rhosc^{n+1})\,\pa_x\rho^{n+1}\,\pa_x[\rhosc^{n+1}]_-\dd x\\+\Delta t\int_\Omega T^{M^\eps_\infty}(\rhosc^{n+1})[\rhosc^{n+1}]_-q_{{\rm div,s}} \, \big(1-R_{{\rm div,s}}(x)\big)\,\big(1-\overline{R}_{{\rm div,s}}(\rho^{n+1}-\eps)\big)\big(1-\underline{R}_{{\rm div,s}}(c_{\rm b})\big)\dd x\\-\Delta t \int_\Omega T^{M^\eps_\infty}(\rhosc^{n+1})[\rhosc^{n+1}]_- \, q_{{\rm s,p}} \, R_{\rm s, p}(x)\,\dd x.
\end{multline*}
Note that all terms in the r.h.s.~vanishes due to the definition of the operator $T^{M^\eps_\infty}$. Then, we obtain
\begin{equation*}
  \|[\rhosc^{n+1}]_-\|^2_{L^2(\Omega)} - \int_\Omega\rhosc^n\,[\rhosc^{n+1}]_-\,\dd x\leq0,
\end{equation*}
which yields
\begin{equation*}
   \|[\rhosc^{n+1}]_-\|^2_{L^2(\Omega)}\leq0.
\end{equation*}
Therefore, we conclude that $\rhosc^{n+1}\geq 0.$
By respecting the order specified at the beginning of the proof, we can show that all  $(\rho_i^{n+1})_{i\in\T}$ are nonnegative. In addition, by using Lemma~\ref{lemma 1}, we get
\begin{equation}
    0\leq \rhosc^{n+1},\rhopc^{n+1},\rhogc^{n+1},\rhoent^{n+1}\leq \rho^{n+1}\leq M_\infty^\eps.
\end{equation}
This ends the proof of Lemma~\ref{lemma 2}.
\end{proof}

\subsection{Study of the equation on the concentration}

\begin{lemma}\label{lemma 3}
Let the assumptions of Lemma~\ref{lemma 2} hold. Then, there exists a nonnegative function $c_{\rm b}^{n+1} \in H$, see definition~\eqref{def : H} of $H$, weak solution to~\eqref{eq : cb eps delta} which satisfies the following uniform in $\eps$ estimate
\begin{equation}\label{eq cb unif delta}
  \| c_{\rm b}^{n+1}\|^2_{L^2(\Omega)}+\Delta t\,\sigma_{\rm b} \| \pa_x c_{\rm b}^{n+1}\|^2_{L^2(\Omega)}\leq  \| c_{\rm b}^{n}\|^2_{L^2(\Omega)}+2\frac{\Delta t}{\sigma_{\rm b}}  {(M_\infty^\eps)}^2\gamma^2.
\end{equation}
\end{lemma}

\begin{proof}
 For all $c_{\rm b}^{n+1}$ and $\varphi\in H$, we consider the following problem: find $c_{\rm b}^{n+1}\in H$ such that 
 \begin{equation}\label{weak formulation 3}
   a_{\rm b}(c_{\rm b}^{n+1},\varphi)=  L_{\rm b}(\varphi),
 \end{equation}
  where  $ a_{\rm b}(\cdot,\cdot)$ is a bilinear form defined on $H\times H$ as
\begin{equation*}
    a_{\rm b}(c_{\rm b}^{n+1},\varphi)= \int_\Omega c_{\rm b}^{n+1} \varphi\,\dd x +\Delta t\,\sigma_{\rm b}\int_\Omega \pa_xc_{\rm b}^{n+1} \pa_x\varphi\,\dd x,
\end{equation*}
and  $ L_{\rm b}$ is a linear form defined on $H$ as 
\begin{equation*}
     L_{\rm b}(\varphi)=\int_\Omega c_{\rm b}^{n} \,\varphi\,\dd x+\gamma\Delta t\int_\Omega \, \dfrac{c^n_{\rm b}+c_{\rm b}^d}{1+c^n_{\rm b}+c_{\rm b}^d} \, (\rhoent^{n+1} + \rhogc^{n+1})\varphi\,\dd x.
\end{equation*}
First,  it is clear that the form $ L_{\rm b}(.)$ is continuous on $H$ and  $ a_{\rm b}(.,.)$ is continuous and coercive on $(H)^2$. Then, by Lax Milgram's theorem we get the existence of a unique solution $c_{\rm b}^{n+1}\in H$ of \eqref{weak formulation 3}.
Now, let us prove the nonnegativity of $c_{\rm b}^{n+1}$. Since $[c_{\rm b}^{n+1}]_-\in H$, we use it as a test function in \eqref{weak formulation 3}. Then, we have
\begin{multline*}
    \int_\Omega|[c_{\rm b}^{n+1}]_-|^2\dd x+\Delta t \sigma_{\rm b}\int_\Omega|\pa_x[c_{\rm b}^{n+1}]_-|^2\dd x =\int_\Omega c_{\rm b}^n[c_{\rm b}^{n+1}]_-\dd x\\+\gamma\Delta t  \int_\Omega\, \dfrac{c^n_{\rm b}+c_{\rm b}^d}{1+c^n_{\rm b}+c_{\rm b}^d} \, (\rhoent^{n+1} + \rhogc^{n+1})[c_{\rm b}^{n+1}]_-\,\dd x.
\end{multline*}
Thanks to the nonnegativity of $c_{\rm b}^n$, $\rhoent^{n+1}$ and $\rhogc^{n+1}$, we get $c_{\rm b}^{n+1}\geq 0$. 

Now, let us establish~\eqref{eq cb unif delta}. For this purpose, we take $c_{\rm b}^{n+1}\in H$ as a test function in~\eqref{weak formulation 3}. After an integration by parts, we have
\begin{multline*}
     \int_\Omega \left(c_{\rm b}^{n+1}-c_{\rm b}^n\right)c_{\rm b}^{n+1}\,\dd x +\Delta t\,\sigma_{\rm b}\int_\Omega |\pa_x c_{\rm b}^{n+1}|^2\dd x=\gamma\Delta t\int_\Omega \, \dfrac{c^n_{\rm b}+c_{\rm b}^d}{1+c^n_{\rm b}+c_{\rm b}^d} \, (\rhoent^{n+1} + \rhogc^{n+1})c_{\rm b}^{n+1}\,\dd x.
\end{multline*}
Using the formula $a(a-b)\geq\frac{a^2}{2}-\frac{b^2}{2}$ together with $fg\leq \sigma_{\rm b} f^2+\frac{1}{4\sigma_{\rm b}}g^2$, we get
\begin{multline}\label{L_2 priori delta t cb}
  \frac12 \| c_{\rm b}^{n+1}\|^2_{L^2(\Omega)}+\Delta t\,\sigma_{\rm b} \| \pa_x c_{\rm b}^{n+1}\|^2_{L^2(\Omega)}\leq \frac12 \| c_{\rm b}^{n}\|^2_{L^2(\Omega)}+\frac{\Delta t}{\sigma_{\rm b}}  (M_\infty^\eps)^2\gamma^2+\Delta t\,\sigma_{\rm b}\| c_{\rm b}^{n+1}\|^2_{L^2(\Omega)}.
\end{multline}
On the other hand, we have for any $x\in\Omega$ that \begin{align*}
    | c_{\rm b}^{n+1}|^2\leq\left(\int_x^1| \pa_x c_{\rm b}^{n+1}(y)|dy\right)^2\leq (1-x)\| \pa_x c_{\rm b}^{n+1}\|^2_{L^2(\Omega)},
\end{align*}
 so that
 \begin{equation}\label{Poincare}
    \| c_{\rm b}^{n+1}\|^2_{L^2(\Omega)}\leq \frac12\| \pa_x c_{\rm b}^{n+1}\|^2_{L^2(\Omega)}.
\end{equation}
It remains to plug it into the equation \eqref{L_2 priori delta t cb} to obtain \eqref{eq cb unif delta}.
\end{proof}

\subsection{Proof of Proposition~\ref{prop : semi discrete}}
Using recursively Lemma \ref{lemma 1}, Lemma \ref{lemma 2} and Lemma \ref{lemma 3}, we deduce the existence of $\rho^{n}$, $\rhosc^{n}$, $\rhopc^{n}$, $\rhoent^{n}$, $\rhogc^{n}$, and $c_{\rm b}^{n}$ for any $1\leq n\leq N_T$ (weak) solutions to~\eqref{eq : rho eps delta}--\eqref{eq : cb eps delta}. Now, we recall that $\Delta t=T/N_T$ and $t_n=n\Delta t$. For any $n \in \{0,...,N_T-~1\}$ and $i\in\T$, we set the piecewise continuous functions in time   \begin{equation*}
    \rho_i^{ \Delta t}(x,t):=\rho_i^{ n+1}(x), \quad \text { for } t \in\left(t_{n}, t_{n+1}\right],
\end{equation*}
with $\rho_i^{\Delta t}(x, 0):=\rho_i^{\eps,0}(x)$  and the function  \begin{equation*}
    c_{\rm b}^{ \Delta t}(x,t):=c_{\rm b}^{ n+1}(x), \quad \text { for } t \in\left(t_{n}, t_{n+1}\right],
\end{equation*}
with $c_{\rm b}^{\Delta t}(x, 0):=c_{\rm b}^{0}(x)$. 
Let $0< J\leq N_T-1$. Then, $\rho^{\Delta t}$ solves the equation
\begin{multline}\label{eq rho delta}
  \frac1{\Delta t}\int_{\Omega \times (\Delta t,T)}\left(\rho^{ \Delta t}(x,t)-\rho^{ \Delta t}(x,t-\Delta t)\right)\phi(x,t)\,\dd x \,\dd t\\ +\eps \int_{\Omega \times (\Delta t,T)} \pa_x \rho^{\Delta t}(x,t) \, \pa_x \phi(x,t) \, \dd x\dd t + \int_{\Omega \times (\Delta t,T)} \rho^{\Delta t}(x,t-\Delta t) \, \pa_x \rho^{\Delta t}(x,t) \, \pa_x \phi(x,t) \, \dd x \dd t \\
= \int_{\Omega \times (\Delta t,T)} f\left(x,\rho^{\Delta t}(x,t)-\eps,\rhosc^{\Delta t}(x,t-\Delta t),\rhopc^{\Delta t}(x,t-\Delta t),\rhoent^{\Delta t}(x,t-\Delta t),\right.\\\left.\rhogc^{\Delta t}(x,t-\Delta t),c_{\rm b}^{\Delta t}(x,t-\Delta t)) \, \phi(x,t) \right)\, \dd x \dd t,
\end{multline}
for piecewise constant functions $\phi:(0,T)\to H^1(\Omega).$ It follows from \cite[Proposition~1.36]{roubivcek2005nonlinear}, that this set of functions is dense in $L^2(0,T;H^1(\Omega))$.

Assuming that $\Delta t < 1$, we sum over $n$ equation~\eqref{coercivity-L2 rho unif delta} and we get 
\begin{equation*}
   \sum_{n=0}^{J}\left( \|\rho^{n+1}\|^2_{L^2(\Omega)}-\|\rho^{n}\|^2_{L^2(\Omega)}\right)+2\eps \,\|\pa_x\rho^{\Delta t}\|^2_{L^2(\Omega\times(0,t_{J+1}))}\leq 8\,T\,q_\infty (M_\infty^\eps)^2(1+2\,q_\infty).
\end{equation*}
Note that $\rho^{J+1}(x)=\rho^{\Delta t}(x,t) $ for $ t\in\left(t_{J}, t_{J+1}\right]$. Then, we have
\begin{equation}\label{unifdeltat rho}
     \|\rho^{\Delta t}(.,t)\|^2_{L^2(\Omega)}+2\eps\, \|\pa_x\rho^{\Delta t}\|^2_{L^2(\Omega\times(0,t_{J+1}))}\leq \|\rho^{\eps,0}\|^2_{L^2(\Omega)}+8\,T\,q_\infty\, (M_\infty^\eps)^2(1+2\,{q_\infty}).
\end{equation}
Since it is true for any $0\leq J\leq N_T-1$,  we can see that $\rho^{\Delta t}$ is uniformly bounded w.r.t.~$\Delta t$ in $L^2(0,T;H^1(\Omega))$.

 Now, let us derive a uniform estimate for the discrete time derivative of $\rho^{\Delta t}$. For this purpose, we take $\varphi\in L^2(0,T;H^1(\Omega))$ and use~\eqref{eq rho delta}. Then, we have 
\begin{align*}
&\frac1{\Delta t}\left|\int_{\Delta t}^T\int_\Omega\left(\rho^{ \Delta t}(x,t)-\rho^{ \Delta t}(x,t-\Delta t)\right)\varphi\,\dd x \,\dd t\right|
= \left|\sum_{n=1}^{N_T-1}\int_{t_n}^{t_{n+1}}\int_\Omega\frac{\rho^{ n+1}-\rho^{ n}}{\Delta t}\,\varphi\,\dd x \,\dd t\right| \\
&\leq  \sum_{n=1}^{N_T-1}\left((\eps+M_\infty^\eps)\int_{t_n}^{t_{n+1}}\|\pa_x\rho^{ n+1}\|_{L^2(\Omega)}\|\pa_x\varphi\|_{L^2(\Omega)}\,\dd t+\int_{t_n}^{t_{n+1}}\|f\|_{L^2(\Omega)}\|\varphi\|_{L^2(\Omega)}\dd t\right).
\end{align*}
Hence,
\begin{multline*}
\frac1{\Delta t}\left|\int_{\Delta t}^T\int_\Omega\left(\rho^{ \Delta t}(x,t)-\rho^{ \Delta t}(x,t-\Delta t)\right)\varphi\,\dd x \,\dd t\right|\\
\leq (\eps+M_\infty^\eps)\left(\sum_{n=1}^{N_T-1}\int_{t_n}^{t_{n+1}}\|\pa_x\rho^{ n+1}\|^2_{L^2(\Omega)}\dd t\right)^{\frac12}\|\pa_x\varphi\|_{L^2(\Omega\times(\Delta t,T))}\\
+4 q_\infty\, M_\infty^\eps\,T^{\frac12}\,\|\varphi\|_{L^2(\Omega\times(\Delta t,T))}, 
\end{multline*}
so that
\begin{multline*}
\frac1{\Delta t}\left|\int_{\Delta t}^T\int_\Omega\left(\rho^{ \Delta t}(x,t)-\rho^{ \Delta t}(x,t-\Delta t)\right)\varphi\,\dd x \,\dd t\right|\\
\leq(\eps+M_\infty^\eps) \|\pa_x\rho^{\Delta t}\|_{L^2(\Omega\times(0,T))}\|\pa_x\varphi\|_{L^2(\Omega\times(0,T))}+4\,q_\infty\, M_\infty^\eps\,T^{\frac12}\,\|\varphi\|_{L^2(\Omega\times(0,T))}.
\end{multline*}
Therefore, we deduce the existence of a constant $C$ independent of $\Delta t$ such that 
\begin{equation*}
    \frac1{\Delta t}\|\rho^{ \Delta t}-\rho^{ \Delta t}(.,.-\Delta t)\|_{L^2(\Delta t,T;(H^1(\Omega))')}\leq C.
\end{equation*}
It remains to apply Aubin–Lions lemma in the version of \cite[Theorem 1]{MR2890969}, 
to deduce that $(\rho^{\Delta t})_{\Delta t}$ is relatively compact in $L^2(\Omega\times(0,T))$ and that there exists $\rho^\eps\in L^2(0,T;H^1(\Omega))$ such that, up to a subsequence, we get
\begin{gather*}
     \pa_x\rho^{\Delta t}\rightharpoonup \pa_x\rho^\eps \quad\text{ weakly in }L^2(\Omega\times(0,T)),\\
    \rho^{\Delta t} \longrightarrow \rho^\eps \quad\text{ strongly in } L^2(\Omega\times(0,T)).
\end{gather*}
In addition, we have 
\begin{equation*}
    \frac{\rho^{n+1}-\rho^{n}}{\Delta t}=\frac{\rho^{\Delta t}(\cdot+\Delta t)-\rho^{\Delta t}(\cdot)}{\Delta t}\rightharpoonup \pa_t\rho^\eps, \quad\text{ as $\Delta t \to 0$, weakly in } L^2(t_{n-1},t_n;(H^1(\Omega))').
    \end{equation*}
Finally, by passing to the limit in \eqref{eq rho delta}, as $\Delta t\to 0$ and for every $\varphi\in L^2(0,T;H^1(\Omega))$, we obtain 
\begin{multline}
\int_{\Omega \times (0,T)} \pa_t \rho^\eps \, \varphi \, \dd x \dd t + \eps \int_{\Omega \times (0,T)} \pa_x \rho^\eps \, \pa_x \varphi \, \dd x\dd t + \int_{\Omega \times (0,T)} \rho^\eps \, \pa_x \rho^\eps \, \pa_x \varphi \, \dd x \dd t \\ 
= \int_{\Omega \times (0,T)} f(x,\rho^\eps-\eps,\rhose,\rhope,\rhoee,\rhoge,c_{\rm b}^\eps) \, \varphi \, \dd x \dd t.
\end{multline}
Similarly, we sum the estimate \eqref{L2 rhoi unif deltat}.
Then, for all $t\in(t_J,t_{J+1}]$ and $\Delta t<1$, we have  
\begin{multline}\label{unif deltat rhosc}
     \|\rhosc^{\Delta t}(.,t)\|^2_{L^2(\Omega)}+\eps \|\pa_x\rhosc^{\Delta t}\|^2_{L^2(\Omega\times(0,t_{J+1}))}\leq \|\rhosc^{\eps,0}\|^2_{L^2(\Omega)}\\+\frac{(M_\infty^\eps)^2}\eps  \|\pa_x\rho^{\Delta t}\|^2_{L^2(\Omega\times(0,T))}+8Tq_\infty (M_\infty^\eps)^2(1+2{q_\infty}).
\end{multline}
Since it is true for any $0\leq J\leq N_T-1$,  we can see that $\rhosc^{\Delta t}$ is uniformly bounded in $\Delta t$ in $L^2(0,T;H^1(\Omega))$.

Furthermore, the control in time is obtained by following the same approach as previously, using \eqref{weak rhosc}. In fact, we have
\begin{multline*}
\frac1{\Delta t}\left|\int_{\Delta t}^T\int_\Omega\left(\rhosc^{ \Delta t}(x,t)-\rhosc^{ \Delta t}(x,t-\Delta t)\right)\varphi\,\dd x \,\dd t\right|\leq\eps\|\pa_x\rhosc^{\Delta t}\|_{L^2(\Omega\times(0,T))}\|\pa_x\varphi\|_{L^2(\Omega\times(0,T))}\\+M_\infty^\eps\|\pa_x\rho^{\Delta t}\|_{L^2(\Omega\times(0,T))}\|\pa_x\varphi\|_{L^2(\Omega\times(0,T))}+2q_\infty M_\infty^\eps\|\varphi\|_{L^2(\Omega\times(0,T))}\leq C,
\end{multline*}
where $C$ is independent of $\Delta t$.
Therefore, we obtain 
\begin{equation*}
    \frac1{\Delta t}\|\rhosc^{ \Delta t}-\rhosc^{ \Delta t}(.,.-\Delta t)\|_{L^2(\Delta t,T;(H^1(\Omega))')}\leq C.
\end{equation*}
Applying once more  Aubin–Lions lemma in the version of \cite[Theorem 1]{MR2890969}, we deduce that $(\rhosc^{\Delta t})_{\Delta t}$ is relatively compact in $L^2(0,T;L^2(\Omega))$.
Finally, by passing to the limit, as $\Delta t\to 0$, we obtain  the existence of $\rhose$ satisfiying~\eqref{weak eq rhoi eps} and similarly for the other densities $\rho_i^\eps$.

Eventually, applying the same strategy to \(c_{\rm b}^{n+1}\), the weak solution of \eqref{eq : cb eps delta}, and then letting \(\Delta t \to 0\), we obtain the existence of $c_{\rm b}^\eps \in L^2(0,T;H)$ which  satisfies \eqref{weak eq cb eps}. This completes the proof of Proposition~\ref{prop : semi discrete}.


\section{Existence proof}\label{sec : exi proof}

In this section we prove Theorem~\ref{thm1}. For this purpose, we first establish uniform with respect to $\eps$ estimates.


\subsection{Uniform estimates}\label{subsec : unif}

\begin{proposition}\label{propistion 2}
Let the assumptions of Theorem~\ref{thm1} hold. Then, there exist a constant $C>0$, only depending on $c^0_b, c_{\rm b}^d, M_\infty, q_\infty, \|\pa_x\rho^0\|_{L^2(\Omega)}$ and $T$ such that
\begin{align} \label{eq : unif eps}
\sum_{i \in \T} \|\pa_x\rho^\eps_i\|_{L^\infty(0,T;L^1(\Omega))} + \|\pa_x \rho^\eps\|^2_{L^\infty(0,T;L^2(\Omega))} + \|\pa_x c_{\rm b}^\eps\|^2_{L^2(\Omega\times(0,T))} \leq C.
\end{align}
\end{proposition}
\begin{proof}
First, assuming that $\eps <1$, the uniform bound of $\|\pa_x c_{\rm b}^\eps(t)\|^2_{L^2(\Omega\times(0,T))}$ is a direct consequence of~\eqref{eq cb unif delta}. Let us now establish an $L^2$ uniform in $\eps$ estimate for $\pa_x\rho^\eps$. To this end, we write the equation of $\rho^\eps$
\begin{equation} \label{eq rho eps}
    \pa_t \rho^{\eps} - \eps \pa_x^2 \rho ^{\eps}- \pa_x(\rho^{\eps} \, \pa_x \rho^\eps) = f(x,\rho^{\eps}-\eps,\rhose,\rhope,\rhoee,\rhoge,c_{\rm b}^\eps), \,\,\,\,\qquad \mbox{in }\Omega \times (0,T),
\end{equation}
where we recall the definition \eqref{def : f} of the source term $f$. 

Thanks to the assumption \emph{{\bf (H4)}} and the parabolic regularity of the nondegenerate equation, see for instance Theorem 5 in~\cite[Chapter 7]{MR2597943}, it can be seen that $\rho^\eps\in L^2(0,T;H^3(\Omega))$.
After differentiating \eqref{eq rho eps}  with respect to 
$x$, we then multiply the equation by $\pa_x\rho^\eps\in L^2(0,T;L^2(\Omega))$ to obtain 
\begin{multline*}
    \frac12\frac{d}{dt}\int_\Omega|\pa_x\rho^\eps|^2\dd x+\eps\int_\Omega |\pa_x^2\rho^\eps|^2\dd x =\int_\Omega\pa_x^2(\rho^\eps\pa_x\rho^\eps)\pa_x\rho^\eps\dd x\\+\int_\Omega \pa_x[f(x,\rho^\eps-\eps,\rhose,\rhope,\rhoee,\rhoge,c_{\rm b}^\eps)]\pa_x\rho^\eps\dd x.
\end{multline*}
After an integration by parts, we get 
\begin{multline*}
  \frac12\frac{d}{dt}\int_\Omega|\pa_x\rho^\eps|^2\dd x+\eps\int_\Omega |\pa_x^2\rho^\eps|^2\dd x =-\frac{1}{3}\int_\Omega\pa_x\left((\pa_x\rho^\eps)^3\right) \dd x  -\int_\Omega\rho^\eps|\pa_x^2\rho^\eps|^2\dd x \\-\int_\Omega f(x,\rho^\eps-\eps,\rhose,\rhope,\rhoee,\rhoge,c_{\rm b}^\eps)\pa_x^2\rho^\eps\dd x.
\end{multline*}
Thanks to Young's inequality, we have
\begin{align}\label{borne pour ln}
      \frac12\frac{d}{dt}\int_\Omega|\pa_x\rho^\eps|^2\dd x+\eps\int_\Omega |\pa_x^2\rho^\eps|^2\dd x+ \frac12\int_\Omega\rho^\eps|\pa_x^2\rho^\eps|^2\dd x\leq \frac12 q_\infty^2\int_\Omega\rho^\eps \dd x.
\end{align}
It remains to integrate in time to obtain 
\begin{align}\label{eq : pax rho eps}
\|\pa_x\rho^\eps(t)\|^2_{L^2(\Omega)}+2\eps\|\pa_x^2\rho^\eps\|^2_{L^2(\Omega\times(0,T))}+\int_0^T\int_\Omega\rho^\eps|\pa_x^2\rho^\eps|\dd x\dd t\leq q_\infty^2 TM_\infty^\eps+\|\pa_x\rho^{0}\|^2_{L^2(\Omega)}.
\end{align}
Finally, for $\eps<1,$  we obtain the uniform bound.

Now, in order to prove that $\rho_i^\eps\in BV(\Omega)$ uniformly in $\eps$, we need to establish a $BV$-estimate on $w_i^\eps$ defined as  $w_i^{\eps}=\rho_i^{\eps}/\rho^{\eps}$, for all $i\in\T$. It is clear that  $0\leq w^\eps_i\leq1$ for any $i\in \T$.
Furthermore, meticulous but rather straightforward computations show that $(w^\eps_i)_{i\in\T}$ satisfies
\begin{equation}
\pa_tw^\eps_i=\eps\pa_x^2w^\eps_i+2\eps\pa_xw^\eps_i\pa_x(\ln{(\rho^\eps)})+\pa_xw^\eps_i\pa_x\rho^\eps+\widetilde{f_i}(x,\rho^\eps-\eps,\wse,\wpe,\wee,\wge,c^\eps_{\rm b}),
\end{equation}
where we define the source terms as follow
\begin{align*}
\widetilde\fsc(x,\rho^\eps-\eps,\wse,\wpe,\wee,\wge,c^\eps_{\rm b}) &\coloneq \fsc(x,\rho^\eps-\eps,\wse,c^\eps_{\rm b})-\wse f(x,\rho^\eps-\eps,\wse,\wpe,\wee,\wge,c^\eps_{\rm b}),
\\
\widetilde\fpc(x,\rho^\eps-\eps,\wse,\wpe,\wee,\wge,c^\eps_{\rm b}) &\coloneq \fpc(x,\rho^\eps-\eps,\wse,\wpe,c^\eps_{\rm b})-\wpe f(x,\rho^\eps-\eps,\wse,\wpe,\wee,\wge,c^\eps_{\rm b}),
\\
\widetilde\fent(x,\rho^\eps-\eps,\wse,\wpe,\wee,\wge,c^\eps_{\rm b}) &\coloneq \fent(x,\rho^\eps-\eps,\wpe,\wee,c^\eps_{\rm b})-\wee f(x,\rho^\eps-\eps,\wse,\wpe,\wee,\wge,c^\eps_{\rm b}),
\\
\widetilde\fgc(x,\rho^\eps-\eps,\wse,\wpe,\wee,\wge,c^\eps_{\rm b}) &\coloneq \fgc(x,\rho^\eps-\eps,\wpe,\wge,c^\eps_{\rm b})-\wge f(x,\rho^\eps-\eps,\wse,\wpe,\wee,\wge,c^\eps_{\rm b}).
\end{align*}
Let us now establish the $BV$-estimate on $w_i^\eps$. We first consider $I_\eta(x)\coloneqq\sqrt{x^2+\eta^2}$, for all $0<\eta\leq1$, a regular approximation (as $\eta\to 0$) of the absolute value ($I^{'}_\eta$ approaches the sign function). Moreover, similarly as $\rho^\eps$, we can see that $\rho_i^\eps\in L^2(0,T;H^3(\Omega))$ and since $\rho^\eps>\eps$, $w_i^\eps$ inhibits the same regularity of $\rho_i^\eps$ for all $i\in\T$. Then, we have
\begin{align} \label{eq I eta}
&\frac{d}{dt}\int_\Omega\sum_{i\in\T}
I_\eta(\pa_x w_i^\eps) \dd x=\sum_{i\in\T}\int_\Omega I_\eta^{'}(\pa_x w_i^\eps)\pa_t\pa_x w_i^\eps\dd x\\
&=\sum_{i\in\T}\Bigg(\eps\int_\Omega I_\eta^{'}(\pa_xw_i^\eps)\pa_x^3w_i^\eps \dd x+2\eps\int_\Omega I_\eta^{'}(\pa_x w_i^\eps)\pa_x\left(\pa_x w_i^\eps \pa_x (\ln(\rho^\eps))\right) \dd x\nonumber\\
&+\int_\Omega I_\eta^{'}(\pa_xw_i^\eps)\pa_x\left(\pa_x w_i^\eps \pa_x \rho^\eps\right) \dd x+\int_\Omega I_\eta^{'}(\pa_xw_i^\eps)\pa_x[\widetilde{f_i}(x,\rho^\eps-\eps,\wse,\wpe,\wee,\wge,c^\eps_{\rm b}) ]\dd x\Bigg).\nonumber
\end{align}
We now treat the terms of the above equation. The first term in the r.h.s.~is nonpositive. Indeed, by convexity of $I_\eta$  and after an integration by parts, we have
\begin{equation*}
    \eps\int_\Omega I_\eta^{'}(\pa_xw_i^\eps)\pa_x^3w_i^\eps \dd x=-\eps\int_\Omega I_\eta^{''}(\pa_xw_i^\eps)|\pa_x^2w_i^\eps|^2 \dd x\leq 0.
\end{equation*}
Now, since the second and third terms on the r.h.s.~are handled in the same way, we detail only the argument for the first one; the second follows analogously:
\begin{multline*}
    \int_\Omega I_\eta^{'}(\pa_xw_i^\eps)\pa_x\left(\pa_x w_i^\eps \pa_x (\ln(\rho^\eps))\right) \,\dd x=\int_\Omega I_\eta^{'}(\pa_xw_i^\eps)\left[\pa_x^2 w_i^\eps \pa_x (\ln(\rho^\eps)) +\pa_x w_i^\eps \pa_x^2 (\ln(\rho^\eps))\right] \dd x\\=\int_\Omega \pa_x[I_\eta(\pa_xw_i^\eps)]\pa_x (\ln(\rho^\eps)) \,\dd x+\int_\Omega \left( I_\eta^{'}(\pa_xw_i^\eps)\pa_x w_i^\eps-I_\eta(\pa_xw_i^\eps)\right) \pa_x^2 (\ln(\rho^\eps)) \,\dd x\\+\int_\Omega I_\eta(\pa_xw_i^\eps)\pa_x^2 (\ln(\rho^\eps))\, \dd x.
\end{multline*}
Note that $$ I_\eta^{'}(\pa_xw_i^\eps)\pa_x w_i^\eps-I_\eta(\pa_xw_i^\eps)= -\frac{\eta^2}{\sqrt{(\pa_x w_i^\eps)^2+\eta^2}}.$$
Then, we have 
\begin{multline*}
     \int_\Omega I_\eta^{'}(\pa_xw_i^\eps)\pa_x\left(\pa_x w_i^\eps \pa_x (\ln(\rho^\eps))\right) \dd x=\int_\Omega \pa_x\left(I_\eta(\pa_x w_i^\eps) \pa_x (\ln(\rho^\eps))\right) \dd x\\-\int_\Omega \frac{\eta^2\left(\rho^\eps\pa_x^2\rho^\eps-(\pa_x\rho^\eps)^2\right)}{(\rho^\eps)^2\sqrt{(\pa_x w_i^\eps)^2+\eta^2}}\dd x.
\end{multline*}
Moreover, thanks to the boundary condition  \eqref{boundary rho eps}, the first term in the r.h.s.~vanishes and  we denoted the second term by $\eta^2R_\eps$. Hence, we obtain
\begin{equation*}
    \int_\Omega I_\eta^{'}(\pa_xw_i^\eps)\pa_x\left(\pa_x w_i^\eps \pa_x (\ln(\rho^\eps))\right) \dd x\leq\eta^2 R_\eps.
\end{equation*}
Back to the equation \eqref{eq I eta} and knowing that $I_\eta^{'}(\pa_x w_i^\eps)\leq1$, it follows that
\begin{equation*}\frac{d}{dt}\int_\Omega\sum_{i\in\T}
I_\eta(\pa_x w_i^\eps) \dd x\leq \eta^2\, |R_\eps| +\int_\Omega\sum_{i\in\T} |\pa_x[\widetilde{f_i}(x,\rho^\eps-\eps,\wse,\wpe,\wee,\wge,c^\eps_{\rm b})]| \dd x,
\end{equation*}
After an integration on $(0,t)$, we get 
\begin{multline*}
\int_\Omega\sum_{i\in\T}
I_\eta(\pa_x w_i^\eps)(t) \dd x\leq \int_\Omega\sum_{i\in\T}
I_\eta(\pa_x w_i^{\eps,0}) \,\dd x +\eta^2\int_0^t |R_\eps|\dd s \\+\int_0^t\int_\Omega\sum_{i\in\T} |\pa_x[\widetilde{f_i}(x,\rho^\eps-\eps,\wse,\wpe,\wee,\wge,c^\eps_{\rm b})]| \dd x\,\dd s.
\end{multline*}
Thanks to \eqref{eq : pax rho eps}, the second term in the r.h.s. goes to $0$ as $\eta\to 0$. Thus, by passing to the limit as $ \eta \to0$, we obtain 
\begin{multline*}
   \sum_{i\in\T}\int_\Omega|\pa_x w_i^\eps|(t)\,\dd x\leq \int_\Omega\sum_{i\in\T}
|\pa_x w_i^{\eps,0}|\, \dd x\\+\int_0^t\int_\Omega \sum_{i\in\T}|\pa_x[\widetilde{f_i}(x,\rho^\eps-\eps,\wse,\wpe,\wee,\wge,c^\eps_{\rm b})]| \dd x\,\dd s.
\end{multline*}
Now, since $w^\eps(\cdot,t)\in W^{1,1}(\Omega),$ 
we have 
\begin{multline*}
\sum_{i\in\T}TV(w_i^\eps)(t)=\sum_{i\in\T}\int_\Omega|\pa_x w_i^\eps(t)|\dd x\leq\sum_{i\in\T} TV( w_i^{\eps,0})\\+\int_0^t\int_\Omega\sum_{i\in\T}\Big(\pa_1\widetilde
     f_i+\pa_x\rho^\eps\pa_2\widetilde
     f_i +\pa_x\wse\pa_3\widetilde
     f_i +\pa_x\wpe\pa_4\widetilde f_i\\
     +\pa_x\wee\pa_5\widetilde
     f_i+\pa_x\wge\pa_6\widetilde
     f_i+\pa_xc_{\rm b}^\eps \pa_7\widetilde
     f_i\Big)\dd x\dd s,
\end{multline*}
where, for $j\in\{1,..7\}$, $\pa_j$ denotes the derivative with respect to each component of $\widetilde f_i$. Moreover, note that all $\pa_j\widetilde f_i$ are bounded, then we get 
\begin{align*}
      \sum_{i\in\T} TV( w_i^\eps)(t)  \leq \sum_{i\in\T}TV( w_i^{\eps,0})+  C\, q_\infty \left(\int_0^t\sum_{i\in\T} TV( w_i^\eps)\dd s+\int_0^t\int_\Omega |\pa_x\rho^\eps|+|\pa_x c_{\rm b}^\eps|\dd x\,\dd s +K\right),
\end{align*}
where $C$ is a constant independent of $\eps$ which depends on $\|R_{\cdot,\cdot}\|_{W^{1,\infty}}$, $\|\overline{R}_{\cdot,\cdot}\|_{W^{1,\infty}}$, and $\|\underline{R}_{\cdot,\cdot}\|_{W^{1,\infty}}$ and $K$ depends on $T$ and $\pa_1 \widetilde{f_i}  $ that means depends on  $\|R_{\cdot,\cdot}\|_{W^{1,\infty}}$. Moreover, using Cauchy--Schwarz inequality, we have  
\[\int_0^T \int_\Omega|\pa_x\rho^\eps| \dd x \dd t\leq T^{1/2}\|\pa_x\rho^\eps\|_{L^2(\Omega\times(0,T))},\]
and similarly for $c_{\rm b}^\eps$. In addition, we denote \[L\coloneq T^{1/2}\left(\|\pa_x\rho^\eps\|_{L^2(\Omega\times(0,T))}+\|\pa_xc^\eps\|_{L^2(\Omega\times(0,T))}\right)+K.\]
After applying Gronwall's lemma, we obtain
\begin{align*}
    \sum_{i\in\T} TV( w_i^\eps)\leq  \left(TV( w_i^{\eps,0})+L\right)\left(1+C\,q_\infty T e^{C q_\infty T}\right).
\end{align*}
Furthermore, for all $i\in\T$ and by the definition~\eqref{def TV} of the semi-norm $TV$, we remark that $TV(w^{\eps,0}_i)$ converges, as $\eps\to 0$, to $ TV(w^{0}_i)$, where we recall that $w_i^0\in BV(\Omega)$ according to assumption \emph{{\bf (H1)}}. Hence, $TV(w_i^{\eps,0})\leq~TV(w_i^0)+~1$ and we get a uniform bound in $\eps$.
Finally, since $\rho_i^\eps=w_i^\eps\rho^\eps$ for all $i\in\T$. Then, $\pa_x\rho_i^\eps=w_i^\eps\pa_x\rho^\eps+\rho^\eps\pa_x w_i^\eps$ is uniformly bounded in $ L^\infty(0,T;L^1(\Omega))$ by using~\eqref{eq : pax rho eps} and the fact that $\pa_x w_i^\eps$ is uniformly bounded in $L^\infty(0,T;L^1(\Omega))$. Therefore, we deduce that the estimate~\eqref{eq : unif eps} holds true.
\end{proof}

\subsection{Proof of Theorem~\ref{thm1}}\label{subsec : limit}
We have shown in Proposition~\ref{prop : semi discrete} and Proposition~\ref{propistion 2}, the existence of $(\rho_i^\eps)_{i\in\T}$ and $c_{\rm b}^\eps$ (weak) solutions of~\eqref{eq : rhosc eps}--\eqref{def : IC cbeps} satisfying the uniform w.r.t.~$\eps$ estimate~\eqref{eq : unif eps}. Hence, for every $\varphi\in~ L^2(0,T;H^1(\Omega))$, we have 
\begin{multline*}
    \left|\int_0^T\int_\Omega \pa_t\rho^\eps\varphi \,dx\,dt\right|\leq\eps\|\pa_x\rho^\eps\|_{L^2(\Omega\times(0,T))}\|\pa_x\varphi\|_{L^2(\Omega\times(0,T))}\\+M_\infty^\eps\|\pa_x\rho^\eps\|_{L^2(\Omega\times(0,T))}\|\pa_x\varphi\|_{L^2(\Omega\times(0,T))}+\|f\|_{L^2(\Omega\times(0,T))}\|\varphi\|_{L^2(\Omega\times(0,T))}.
\end{multline*}
Thus, for $\eps<1,$  we get
\[
\|\pa_t\rho^\eps\|_{L^2(0,T;(H^1(\Omega))')}\leq C.
\]
Consequently, using Aubin-Lions lemma, we obtain, up to a subsequence, as $\eps\to0$, that
\begin{align*}
\pa_x\rho^{\eps}&\rightharpoonup \pa_x\rho \quad\text{ weakly in }L^2(\Omega\times(0,T)),\\
\rho^\eps &\longrightarrow \rho \quad\text{  strongly in } L^2(\Omega\times(0,T)),\\
\pa_t \rho^\eps &\longrightarrow\pa_t \rho \,\text{ weakly in } L^2(0,T;H^1(\Omega)').
\end{align*} 

Now, for any $\varphi\in L^2(0,T;H^2(\Omega))$, we have 
\begin{multline*}
    \left|\int_0^T\int_\Omega \pa_t\rho_i^\eps\varphi \,dx\,dt\right|\leq\eps\, M^\eps_\infty \|\pa^2_x\varphi\|_{L^2(\Omega\times(0,T))}\\+M_\infty^\eps\|\pa_x\rho^\eps\|_{L^2(\Omega\times(0,T))}\|\pa_x\varphi\|_{L^2(\Omega\times(0,T))}+\|f_i\|_{L^2(\Omega\times(0,T))}\|\varphi\|_{L^2(\Omega\times(0,T))}.
\end{multline*}
Hence, for $\eps<1$, we obtain $$\|\pa_t\rho_i^\eps\|_{L^2(0,T;(H^2(\Omega))')}\leq C.$$
In addition, thanks to \eqref{eq : unif eps}, we recall that $\rho_i^\eps$ is uniformly bounded in $L^\infty(0,T;BV(\Omega))$. Moreover, since $BV(\Omega)\hookrightarrow L^p(\Omega)$ compactly for all $1\leq p<+\infty$, we deduce from Aubin–Lions lemma that for any $i \in \T$, up to a subsequence, as $\eps\to0$ it holds 
\begin{align*}
\rho_i^\eps &\longrightarrow\rho_i \quad  \text{ strongly in } L^\infty(0,T;L^p(\Omega)), \quad  \forall 1\leq p<+\infty,\\
\pa_t \rho_i^\eps &\longrightarrow\pa_t \rho_i \,\, \text{ weakly in } L^2(0,T;H^2(\Omega)').
\end{align*}
Hence, $\rho_i^\eps$ converges (up to subsequence) a.e. to $\rho_i$ in $\Omega\times(0,T)$. Moreover, as $\rho^\eps \leq M_\infty^\eps$ a.e. in $\Omega \times(0,T)$, we conclude thanks to Lebesgues dominated convergence theorem that for every $\psi\in L^2(0,T;L^2(\Omega))$, up to a subsequence, we have 
\begin{equation*}
\rho_i^\eps\psi\longrightarrow\rho_i\psi \quad\text{ strongly in } L^2(\Omega\times(0,T)) \mbox{ as } \eps \to 0.
\end{equation*}

Let us now notice that $c_{\rm b}^\eps$ is uniformly bounded, for $\eps<1$, in $L^2(0,T;H)$ by~\eqref{eq : unif eps} and the Poincar\'e inequality (as in~\eqref{Poincare}). Moreover, it is clear that 
 for $\eps<1$, we have $$\|\pa_t c_{\rm b}^\eps\|_{L^2(0,T;H')}\leq C.$$
 Then, applying once more time Aubin–Lions lemma, we obtain, up to a subsequence, as $\eps\to0$ that 
 \begin{align*}
 \pa_t c_{\rm b}^\eps& \longrightarrow \pa_t c_{\rm b} \quad\text{ weakly in } L^2(0,T;H'),\\
     c_{\rm b}^\eps &\longrightarrow c_{\rm b}  \quad\text{ strongly in } L^2(\Omega\times(0,T)),\\
      \dfrac{c^\eps_{\rm b}+c_{\rm b}^d}{1+c^\eps_{\rm b}+c_{\rm b}^d}&\longrightarrow \dfrac{c_{\rm b}+c_{\rm b}^d}{1+c_{\rm b}+c_{\rm b}^d} \quad\text{ strongly in } L^2(\Omega\times(0,T)).
 \end{align*}
 Finally, by assumption \emph{\bf{(H4)}}, all the source terms are Lipschitz continuous with respect to all variables and all these variables converge (up to a subsequence) strongly in  $L^2(\Omega\times(0,T))$. Then, as $\eps\to0$, we have that
\begin{equation*}
    f_i(x,\rho^\eps,(\rho^\eps_j)_{j\in\mathcal{E}_i},c^\eps_{\rm b})\longrightarrow f_i(x,\rho,(\rho_j)_{j\in\mathcal{E}_i},c_{\rm b}) \text{ strongly in } L^2(\Omega\times(0,T)).
\end{equation*}
Therefore, by combining all the previous results and by passing to the limit in \eqref{weak eq rhoi eps}, we get \eqref{ eq rhoi weak final}, where we have used the weak $L^2$- strong $L^2$ convergence in the product $\rho_i^\eps\pa_x\rho^\eps\pa_x\varphi.$ 
Similarly, by passing to the limit in \eqref{weak eq cb eps} we obtain \eqref{ eq cb weak final }. This concludes the proof of Theorem~\ref{thm1}.




\bibliographystyle{plain}
\bibliography{biblio}

@article{DL89,
 author = {Di Perna, R. J. and Lions, P. L.},
 title = {Ordinary differential equations, transport theory and {Sobolev} spaces},
 fjournal = {Inventiones Mathematicae},
 journal = {Invent. Math.},
 issn = {0020-9910},
 volume = {98},
 number = {3},
 pages = {511--547},
 year = {1989},
 language = {English},
 doi = {10.1007/BF01393835},
 url = {https://eudml.org/doc/143741},
 zbMATH = {4140224},
 Zbl = {0696.34049}
}

@article{BGHP85,
 author = {Bertsch, M. and Gurtin, M. E. and Hilhorst, D. and Peletier, L. A.},
 title = {On interacting populations that disperse to avoid crowding: preservation of segregation},
 fjournal = {Journal of Mathematical Biology},
 journal = {J. Math. Biol.},
 issn = {0303-6812},
 volume = {23},
 pages = {1--13},
 year = {1985},
 language = {English},
 doi = {10.1007/BF00276555},
 zbMATH = {3960096},
 Zbl = {0596.35074}
}

@PHDTHESIS{darri2020,
url = "http://www.theses.fr/2020UPASM008",
title = "Modélisation du dialogue hôte-microbiote au voisinage de l’épithélium de l'intestin distal",
author = "Darrigade, L.",
year = "2020",
school = "Université Paris-Saclay",
url = "http://www.theses.fr/2020UPASM008/document",
}

@article{Azais_reviewPDMP14,
 author = {Aza{\"{\i}}s, R. and Bardet, J.-B. and G{\'e}nadot, A. and Krell, N. and Zitt, P.-A.},
 title = {Piecewise deterministic {Markov} process -- recent results},
 fjournal = {European Series in Applied and Industrial Mathematics (ESAIM): Proceedings},
 journal = {ESAIM, Proc.},
 issn = {1270-900X},
 volume = {44},
 pages = {276--290},
 year = {2014},
 language = {English},
 doi = {10.1051/proc/201444017},
 zbMATH = {6510049},
 Zbl = {1331.60153}
}

@article{Davis84,
 author = {Davis, M. H. A.},
 title = {Piecewise-deterministic {Markov} processes: {A} general class of non- diffusion stochastic models},
 fjournal = {Journal of the Royal Statistical Society. Series B},
 journal = {J. R. Stat. Soc., Ser. B},
 issn = {0035-9246},
 volume = {46},
 pages = {353--388},
 year = {1984},
 language = {English},
 zbMATH = {3901778},
 Zbl = {0565.60070}
}

@article{BGH871,
author = {M. Bertsch and M.E. Gurtin and D. Hilhorst},
title = {On interacting populations that disperse to avoid crowding: the case of equal dispersal velocities},
journal = {Nonlinear Analysis: Theory, Methods \& Applications},
volume = {11},
number = {4},
pages = {493-499},
year = {1987},
issn = {0362-546X},
doi = {https://doi.org/10.1016/0362-546X(87)90067-8}
}

@article{BGH872,
 author = {Bertsch, M. and Gurtin, M. E. and Hilhorst, D.},
 title = {On a degenerate diffusion equation of the form {{\(c(z)_ t={{\phi}} (z_ x)_ x\)}} with application to population dynamics},
 fjournal = {Journal of Differential Equations},
 journal = {J. Differ. Equations},
 issn = {0022-0396},
 volume = {67},
 pages = {56--89},
 year = {1987},
 language = {English},
 doi = {10.1016/0022-0396(87)90139-2},
 zbMATH = {4013159},
 Zbl = {0624.35049}
}

@article{BHIM12,
 author = {Bertsch, M. and Hilhorst, D. and Izuhara, H. and Mimura, M.},
 title = {A nonlinear parabolic-hyperbolic system for contact inhibition of cell-growth},
 fjournal = {Differential Equations and Applications},
 journal = {Differ. Equ. Appl.},
 issn = {1847-120X},
 volume = {4},
 number = {1},
 pages = {137--157},
 note = {Id/No 09},
 year = {2012},
 language = {English},
 url = {files.ele-math.com/articles/dea-04-09.pdf},
 zbMATH = {6037191},
 Zbl = {1238.35161}
}

@article{MurTog15,
 author = {Murakawa, H. and Togashi, H.},
 title = {Continuous models for cell-cell adhesion},
 fjournal = {Journal of Theoretical Biology},
 journal = {J. Theor. Biol.},
 issn = {0022-5193},
 volume = {374},
 pages = {1--12},
 year = {2015},
 language = {English},
 doi = {10.1016/j.jtbi.2015.03.002},
 zbMATH = {6606638},
 Zbl = {1341.92019}
}

@article{CHS18,
author = {Carrillo, J. and Huang, Y. and Schmidtchen, M.},
title = {Zoology of a Nonlocal Cross-Diffusion Model for Two Species},
journal = {SIAM Journal on Applied Mathematics},
volume = {78},
number = {2},
pages = {1078-1104},
year = {2018},
doi = {10.1137/17M1128782}
}

@article{ChaLol05,
 author = {Chaplain, M. A. J. and Lolas, G.},
 title = {Mathematical modelling of cancer cell invasion of tissue: the role of the urokinase plasminogen activation system},
 fjournal = {M\(^3\)AS. Mathematical Models \& Methods in Applied Sciences},
 journal = {Math. Models Methods Appl. Sci.},
 issn = {0218-2025},
 volume = {15},
 number = {11},
 pages = {1685--1734},
 year = {2005},
 language = {English},
 doi = {10.1142/S0218202505000947},
 zbMATH = {2243293},
 Zbl = {1094.92039}
}

@book{Preziosi03,
 editor = {Preziosi, Luigi},
 title = {Cancer modelling and simulation},
 isbn = {1-58488-361-8; 978-0-203-49489-9},
 year = {2003},
 publisher = {Boca Raton, FL: Chapman {and} Hall/CRC},
 language = {English},
 doi = {10.1201/9780203494899},
 zbMATH = {1912158},
 Zbl = {1039.92022}
}

@article{APS06,
 author = {Armstrong, N. J. and Painter, K. J. and Sherratt, J. A.},
 title = {A continuum approach to modelling cell-cell adhesion},
 fjournal = {Journal of Theoretical Biology},
 journal = {J. Theor. Biol.},
 issn = {0022-5193},
 volume = {243},
 number = {1},
 pages = {98--113},
 year = {2006},
 language = {English},
 doi = {10.1016/j.jtbi.2006.05.030},
 url = {europepmc.org/articles/pmc1941683},
 zbMATH = {7249364},
 Zbl = {1447.92113}
}

@article{CFSS_JKO_18,
 author = {Carrillo, J. A. and Fagioli, S. and Santambrogio, F. and Schmidtchen, M.},
 title = {Splitting schemes and segregation in reaction cross-diffusion systems},
 fjournal = {SIAM Journal on Mathematical Analysis},
 journal = {SIAM J. Math. Anal.},
 issn = {0036-1410},
 volume = {50},
 number = {5},
 pages = {5695--5718},
 year = {2018},
 language = {English},
 doi = {10.1137/17M1158379},
 zbMATH = {6970868},
 Zbl = {1402.35147}
}

@article{BHIMW20,
 author = {Bertsch, M. and Hilhorst, D. and Izuhara, H. and Mimura, M. and Wakasa, T.},
 title = {A nonlinear parabolic-hyperbolic system for contact inhibition and a degenerate parabolic {Fisher}-{KPP} equation},
 fjournal = {Discrete and Continuous Dynamical Systems},
 journal = {Discrete Contin. Dyn. Syst.},
 issn = {1078-0947},
 volume = {40},
 number = {6},
 pages = {3117--3142},
 year = {2020},
 language = {English},
 doi = {10.3934/dcds.2019226},
 zbMATH = {7183831},
 Zbl = {1435.35126}
}

@article{BPM10,
 author = {Bertsch, M. and Dal Passo, R. and Mimura, M.},
 title = {A free boundary problem arising in a simplified tumour growth model of contact inhibition},
 fjournal = {Interfaces and Free Boundaries},
 journal = {Interfaces Free Bound.},
 issn = {1463-9963},
 volume = {12},
 number = {2},
 pages = {235--250},
 year = {2010},
 language = {English},
 doi = {10.4171/IFB/233},
 zbMATH = {5789981},
 Zbl = {1197.35296}
}

@article{GalSel15,
 author = {Galiano, G. and Selgas, V.},
 title = {Analysis of a splitting-differentiation population model leading to cross-diffusion},
 fjournal = {Computers \& Mathematics with Applications},
 journal = {Comput. Math. Appl.},
 issn = {0898-1221},
 volume = {70},
 number = {12},
 pages = {2933--2945},
 year = {2015},
 language = {English},
 doi = {10.1016/j.camwa.2015.10.005},
 zbMATH = {7258727},
 Zbl = {1443.92152}
}

@article{DHCLL22,
 author = {Darrigade, L. and Haghebaert, M. and Cherbuy, C. and Labarthe, S. and Laroche, B.},
 title = {A {PDMP} model of the epithelial cell turn-over in the intestinal crypt including microbiota-derived regulations},
 fjournal = {Journal of Mathematical Biology},
 journal = {J. Math. Biol.},
 issn = {0303-6812},
 volume = {84},
 number = {7},
 pages = {67},
 note = {Id/No 60},
 year = {2022},
 language = {English},
 doi = {10.1007/s00285-022-01766-8},
 zbMATH = {7556009},
 Zbl = {1496.92015}
}

@article{Jue15,
 author = {J{\"u}ngel, A.},
 title = {The boundedness-by-entropy method for cross-diffusion systems},
 fjournal = {Nonlinearity},
 journal = {Nonlinearity},
 issn = {0951-7715},
 volume = {28},
 number = {6},
 pages = {1963--2001},
 year = {2015},
 language = {English},
 doi = {10.1088/0951-7715/28/6/1963},
 zbMATH = {6465032},
 Zbl = {1326.35175}
}

@article{HoJu25,
 author = {Hopf, K. and J{\"u}ngel, A.},
 title = {Convergence of a finite-volume scheme and dissipative measure-valued-strong stability for a hyperbolic-parabolic cross-diffusion system},
 fjournal = {Numerische Mathematik},
 journal = {Numer. Math.},
 issn = {0029-599X},
 volume = {157},
 number = {3},
 pages = {951--992},
 year = {2025},
 language = {English},
 doi = {10.1007/s00211-025-01474-7},
 zbMATH = {8052718}
}

@article{DrHoJu23,
 author = {Druet, P.-E. and Hopf, K. and J{\"u}ngel, A.},
 title = {Hyperbolic-parabolic normal form and local classical solutions for cross-diffusion systems with incomplete diffusion},
 fjournal = {Communications in Partial Differential Equations},
 journal = {Commun. Partial Differ. Equations},
 issn = {0360-5302},
 volume = {48},
 number = {6},
 pages = {863--894},
 year = {2023},
 language = {English},
 doi = {10.1080/03605302.2023.2212479},
 zbMATH = {7755393},
 Zbl = {1526.35243}
}

@article{DrJu20,
 author = {Druet, P.-E. and J{\"u}ngel, A.},
 title = {Analysis of cross-diffusion systems for fluid mixtures driven by a pressure gradient},
 fjournal = {SIAM Journal on Mathematical Analysis},
 journal = {SIAM J. Math. Anal.},
 issn = {0036-1410},
 volume = {52},
 number = {2},
 pages = {2179--2197},
 year = {2020},
 language = {English},
 doi = {10.1137/19M1301473},
 zbMATH = {7201607},
 Zbl = {1442.35188}
}

@article{PQV14,
 author = {Perthame, B. and Quir{\'o}s, F. and V{\'a}zquez, J. L.},
 title = {The {Hele}-{Shaw} asymptotics for mechanical models of tumor growth},
 fjournal = {Archive for Rational Mechanics and Analysis},
 journal = {Arch. Ration. Mech. Anal.},
 issn = {0003-9527},
 volume = {212},
 number = {1},
 pages = {93--127},
 year = {2014},
 language = {English},
 doi = {10.1007/s00205-013-0704-y},
 zbMATH = {6303385},
 Zbl = {1293.35347}
}

@article{BPPS20,
 author = {Bubba, F. and Perthame, B. and Pouchol, C. and Schmidtchen, M.},
 title = {Hele-Shaw limit for a system of two reaction-({Cross}-)diffusion equations for living tissues},
 fjournal = {Archive for Rational Mechanics and Analysis},
 journal = {Arch. Ration. Mech. Anal.},
 issn = {0003-9527},
 volume = {236},
 number = {2},
 pages = {735--766},
 year = {2020},
 language = {English},
 doi = {10.1007/s00205-019-01479-1},
 zbMATH = {7178292},
 Zbl = {1435.35390}
}

@article{PerVau15,
 author = {Perthame, B. and Vauchelet, N.},
 title = {Incompressible limit of a mechanical model of tumour growth with viscosity},
 fjournal = {Philosophical Transactions A. Royal Society of London},
 journal = {Philos. Trans. A, R. Soc. Lond.},
 issn = {1364-503X},
 volume = {373},
 number = {2050},
 pages = {16},
 note = {Id/No 20140283},
 year = {2015},
 language = {English},
 doi = {10.1098/rsta.2014.0283},
 zbMATH = {6674737},
 Zbl = {1353.35294}
}

@article{GuPi84,
 author = {Gurtin, Morton E. and Pipkin, A. C.},
 title = {A note on interacting populations that disperse to avoid crowding},
 fjournal = {Quarterly of Applied Mathematics},
 journal = {Q. Appl. Math.},
 issn = {0033-569X},
 volume = {42},
 pages = {87--94},
 year = {1984},
 language = {English},
 doi = {10.1090/qam/736508},
 zbMATH = {3849040},
 Zbl = {0534.92021}
}

@article {MR2890969,
    AUTHOR = {Dreher, Michael and J\"ungel, Ansgar},
     TITLE = {Compact families of piecewise constant functions in
              {$L^p(0,T;B)$}},
   JOURNAL = {Nonlinear Anal.},
  FJOURNAL = {Nonlinear Analysis. Theory, Methods \& Applications. An
              International Multidisciplinary Journal},
    VOLUME = {75},
      YEAR = {2012},
    NUMBER = {6},
     PAGES = {3072--3077},
      ISSN = {0362-546X,1873-5215},
   MRCLASS = {46B50 (35A15 35A35)},
  MRNUMBER = {2890969},
MRREVIEWER = {Narcisse\ Randrianantoanina},
       DOI = {10.1016/j.na.2011.12.004},
       URL = {https://doi.org/10.1016/j.na.2011.12.004},
}

@book {MR2597943,
    AUTHOR = {Evans, Lawrence C.},
     TITLE = {Partial differential equations},
    SERIES = {Graduate Studies in Mathematics},
    VOLUME = {19},
   EDITION = {Second},
 PUBLISHER = {American Mathematical Society, Providence, RI},
      YEAR = {2010},
     PAGES = {xxii+749},
      ISBN = {978-0-8218-4974-3},
   MRCLASS = {35-01},
  MRNUMBER = {2597943},
MRREVIEWER = {Diego\ M.\ Maldonado},
       DOI = {10.1090/gsm/019},
       URL = {https://doi.org/10.1090/gsm/019},
}

@book{roubivcek2005nonlinear,
  title={Nonlinear partial differential equations with applications},
  author={Roub{\'\i}{\v{c}}ek, T.},
  year={2005},
  publisher={Springer}
}

\end{document}